\documentclass[12pt,espf]{amsart}
\usepackage{a4}
\usepackage{hyperref}
\usepackage{amssymb}
\usepackage{amsmath}
\usepackage{amsthm}
\usepackage{amstext}
\usepackage{amscd}
\usepackage{latexsym}
\usepackage{graphics}
\usepackage{color}
\usepackage{url}

\theoremstyle{plain}
\newtheorem{thm}{Theorem}[subsection]
\newtheorem{theorem}[thm]{Theorem}

\newtheorem{lemma}[thm]{Lemma}
\newtheorem{lem}[thm]{Lemma}
\newtheorem{proposition}[thm]{Proposition}
\newtheorem{cor}[thm]{Corollary}
\newtheorem{corollary}[thm]{Corollary}

\newtheorem{q}[thm]{Question}
\theoremstyle{definition}
\newtheorem{rem}[thm]{Remark}

\newtheorem{rmk}[thm]{Remark}

\newtheorem{defn}[thm]{Definition}
\newtheorem{ex}[thm]{Example}

\newtheorem{notation}[thm]{Notation}
\newtheorem{Definition-Proposition}[thm]{Definition-Proposition}

\newcommand{\sA}{{\mathcal A}}
\newcommand{\sB}{{\mathcal B}}

\newcommand{\sH}{{\mathcal H}}
\newcommand{\sF}{{\mathcal F}}

\newcommand{\sL}{{\mathcal L}}
\newcommand{\sC}{{\mathcal C}}
\newcommand{\sD}{{\mathcal D}}
\newcommand{\sI}{{\mathcal I}}
\newcommand{\sT}{{\mathcal T}}
\newcommand{\sO}{{\mathcal O}}
\newcommand{\sP}{{\mathcal P}}
\newcommand{\bC}{{\mathbb C}}

\newcommand{\bP}{{\mathbb P}}
\newcommand{\bQ}{{\mathbb Q}}
\newcommand{\bZ}{{\mathbb Z}}

\DeclareMathOperator{\hocolim}{\underset{\longrightarrow}{hocolim}}
\DeclareMathOperator{\holim}{\underset{\longleftarrow}{holim}}

\DeclareMathOperator{\coker}{{coker}}

\DeclareMathOperator{\Alb}{{Alb}}
\DeclareMathOperator{\im}{{im}}
\DeclareMathOperator{\rk}{{rk}}
\DeclareMathOperator{\id}{{id}}
\DeclareMathOperator{\Pic}{{Pic}}
\DeclareMathOperator{\Supp}{{Supp}}

\DeclareMathOperator{\Spec}{{Spec}}
\DeclareMathOperator{\Id}{{Id}}
\input{xy}
\xyoption{all}

\author{Christopher D. Hacon}
\address{Department of Mathematics, University of Utah, 155 South 1400 East,
Salt Lake City, UT 48112-0090, USA}
\email{hacon@math.utah.edu}
\author{Zsolt Patakfalvi}
\address{Department of Mathematics, Princeton University, Fine Hall, Washington Road,  NJ 08544-1000, USA}
\email{pzs@math.princeton.edu}

\begin{document}
\title[Generic vanishing and ordinary abelian varieties]{Generic vanishing in characteristic $p>0$ and the characterization of ordinary abelian varieties}
\maketitle

\begin{abstract}
We prove a generic vanishing type statement in positive characteristic and apply it to prove positive characteristic versions of Kawamata's theorems: a characterization of smooth varieties birational to ordinary abelian varieties and the surjectivity of the Albanese map when the Frobenius stable Kodaira dimension is zero.
\end{abstract}

\section{Introduction}
\label{sec:introduction}

\subsection{Varieties of maximal Albanese dimension.}
Let $X$ be a smooth projective variety over an algebraically closed field $k$.
If $X$ admits a generically finite morphism to an abelian variety $f:X\to A$ then we say that $X$  has maximal Albanese dimension (m.A.d.).
The geometry of complex projective m.A.d. varieties is extremely well understood. In particular it is known that these varieties admit a good minimal model (cf. \cite{Fujino09}), and that the 4-th pluricanonical map gives the Iitaka fibration (cf. \cite{JLT11}). The main tool used in proving results about the geometry of  m.A.d. varieties are the generic vanishing theorems first developed by Green and Lazarsfeld (and further sharpened by Chen, Hacon, Popa, Pareschi, Schnell, Simpson and others). On the other hand, in positive characteristic  very little is known about generic vanishing and m.A.d. varieties. By a result of the first author and Kov\'acs it is known that the obvious positive characteristic version of the generic vanishing theorem does not hold \cite{HK12}. Here we present a generic vanishing statement in positive characteristic which, by the results of \cite{HK12} is necessarily weaker than the characteristic zero statements, but it is strong enough to imply the following positive 
characteristic 
version of the celebrated results of Kawamata \cite{Kawamata81}.
\begin{thm}\label{t-one}
 Let $X$ be a smooth projective variety  over an algebraically closed field $k$ of characteristic $p>0$, and let $a:X\to A$ be the Albanese morphism. 
\begin{enumerate}
\item If $\kappa_S(X)=0$, then $a$ is surjective and in particular we obtain the upper bound $b_1(X) \leq 2  \dim X$ for the first Betti number.
\item $X$ is birational to an ordinary abelian variety if and only if  $p \nmid \deg a$,  $\kappa_S(X)=0$ and $b_1(X)=2\dim (X)$.\end{enumerate}
\end{thm}
Here $b_1(X)$ is the first Betti-number, which by definition is $\dim_{\bQ_l} H^1_{\textrm{\'et}} (X, \bQ_l)$ for any $l \neq p$, and $\kappa_S(X)$ is the Frobenius stable Kodaira dimension (see the later parts of the introduction or the beginning of \S 4 for the definition of $\kappa_S(X)$). Note also that it is well known that $b_1(X)= 2\dim A$ \cite[page 14]{Liedtke}. 

\subsection{Generic vanishing over the complex numbers.}
To explain our positive characteristic generic vanishing statement, let us recall briefly the known results on generic vanishing over the complex numbers. 
Let $X$ be a smooth projective variety over $\bC$ of dimension $d$, let $a:X\to A$ be the Albanese morphism and $$V^i(\omega _X):=\{P\in {\rm Pic}^0(X)|h^i(X,\omega _X\otimes P)\ne 0\}$$ the cohomology support loci.
\begin{thm}\label{t-gl} \cite{GL91,Simpson93} Every irreducible component of $V^i(\omega _X)$ is a (torsion) translate of a (reduced) subtorus of ${\rm Pic}^0(X)$ of codimension at least $i-({\rm dim }(X)-\dim (a(X)))$. If $\dim (X)=\dim\ a(X)$ then there are inclusions:
$$V^0(\omega _X)\supset V^1(\omega _X )\supset \ldots \supset V^{\dim (X)}(\omega _X)=\{\mathcal O_X\}.$$\end{thm}
In particular if $X$ is a variety of m.A.d., then the $V^i(\omega _X)$ have codimension $\geq i$, which implies the vanishing of  $H^i(X, \omega_X \otimes P)$ for $i>0$ and generic $P \in {\rm Pic}^0(X)$. Hence  this result is known as a "generic vanishing theorem". It implies that for m.A.d. varieties one can replace the ample line bundle in the statement of Kodaira vanishing by a general topologically trivial line bundle. In fact, \eqref{t-gl} can be thought of as the limit of  Kodaira vanishing if one considers the alternative point of view of \cite{Hacon04} and \cite{PP09}. In these articles  the following seemingly unrelated conjecture of Green and Lazarsfeld was proven. 
\begin{thm}\label{t-dcgv} \cite{Hacon04,PP09} Let $\sL$ be the Poincar\'e line bundle on $A\times \hat A$, where $\hat A := {\rm Pic}^0(A) (\cong {\rm Pic}^0 (X))$ is the dual abelian variety of $A$. If the Albanese image $a(X)\subset A$ has dimension $d-k$, then $$R^i{p_{\hat A,*}} (\sL _X )=0$$ for $i\not\in [d-k, d]$ and $\sL _X :=(a\times {\rm id}_{\hat A})^* \sL$.
\end{thm}
It turns out that in fact \eqref{t-dcgv} and \eqref{t-gl} are equivalent by \cite{PP11}.  To explain this, note that by applying standard derived category machinery (Grothendieck spectral sequence, projection formula and Grothendieck duality), one can show that $R p_{\hat A,*}  ( \sL_X) \cong R \hat S (D_A(Ra_* \omega_X[d]))$, where $D_A(?) \cong R \mathcal{H}om (?, \sO _A[g])$, $g=\dim A$  and $R\hat S:D(A)\to D(\hat A)$ is the Fourier-Mukai functor defined in \cite{Mukai81} so that 
$R\hat S(?)=Rp_{\hat A,*}(Lp^*_A(?)\otimes \sL )$. Further, by a famous result of Koll\'ar, in characteristic zero $Ra_{*}\omega _X=\sum _{j=0}^kR^ja_{*}\omega _X [-j]$ (cf. \cite{Kol86}). Thus, \eqref{t-dcgv} 
is equivalent to the vanishing  
$$\mathcal H^i(R \hat S (D_A(R^j a_* \omega_X)))=0$$
for each $0  \leq j \leq k$ and $i < j-k$. 

The latter condition was shown in \cite{PP11} to be equivalent to \eqref{t-gl}. Let $\phi _L:\hat A\to A$ be the corresponding isogeny determined by the formula $\phi _L(x)=t_x^*L\otimes L^\vee$, where $t_x$ is the translation by $x\in \hat A$. Then the dual $RS(L)=\sH^0(RS(L))=\hat L$ of $L$ is a vector bundle on $A$ of rank $h^0(L)$ such that $\phi _L^*(\hat L)=\bigoplus _{h^0(L)}L^\vee$ (here $RS : D(\hat A) \to D(A)$ is the inverse Fourier-Mukai functor $R S(?)=Rp_{ A,*}(Lp^*_{\hat A}(?)\otimes \sL )$ ).

\begin{thm}\cite[Theorem A]{PP11}\label{t-eqGV} Let $A$ be an abelian variety  over an algebraically closed field and $F$ a coherent sheaf on $A$ and $l \geq 0$ an integer, then the following are equivalent:
\begin{enumerate}
\item ${\rm codim} \left( V^i(F):=\{ P\in \hat A |h^i(F\otimes P)\ne 0\}\right) \geq i-l$
\item $\sH^i(R\hat S (D_A(F)))=0$ for \ all $i<-l$.
\item $H^i(A,F\otimes \hat L ^\vee )=0$ for $i>l$ and any sufficiently ample line bundle $L$.
\end{enumerate}\end{thm}
Heuristically, we think of $\hat L^\vee $ as an ample vector bundle which  plays the role of $\frac 1 {{\rm deg}(L)}L$, so that as ${\rm deg }(L)$ increases, $\hat L^\vee $ corresponds to a smaller and smaller multiple of an ample line bundle (alternatively the rank of $\hat L^\vee $ increases but $h^0(\hat L^\vee )=1$ is fixed). In this way we interpret, as hinted earlier, the Generic Vanishing Theorem as a limit of the Kodaira vanishing theorem. In fact,  in characteristic $0$, it is easy to see that Koll\'ar vanishing implies that each $R^ja_*\omega _X $ satisfies (3) of \eqref{t-eqGV} with $l= 0$ and hence that  \eqref{t-dcgv} and \eqref{t-gl} hold (cf. \cite{Hacon04}).

\subsection{Generic vanishing in positive characteristic}
\label{subsec:introduction:GVP}
It is a natural question to generalize these important results to positive characteristic. The main issue in doing so is that
in characteristic zero, Koll\'ar's vanishing is used to prove (3) of \eqref{t-eqGV}, while in positive characteristic this vanishing theorem is known to fail. In fact, \cite{HK12} gives some elementary counter examples to generic vanishing in characteristic $p$ and shows that generic vanishing results for a generically finite seperable morphism $a:X\to A$ from a smooth variety to an abelian variety should be equivalent to the vanishing  $R^ia_* \omega _X=0$ for $i>0$. Thus it is clear that the naive generalization to positive characteristic fails and a new approach is necessary.

To establish this new approach, we begin by investigating the fundamental properties of the sheaves $a_* \omega_X$ (and $R^ia_*\omega_X$ for $i\geq 0$). Over the complex numbers it is well known that (under mild technical assumptions)  $ a_* \omega_X $ is the upper canonical extension of the lowest piece in the Hodge filtration of the variation of Hodge structures on $R^k a_* \mathbb{C}_X$ \cite[Theorem 2.6]{Kol86}. In particular, $a_* \omega_X$ is the lowest filtered piece of a filtered $D$-module on $A$. The work of Schnell and Popa on generic vanishing clearly illustrates that (in this context) this is the right way of thinking about $a_* \omega_X$ \cite{PS}. In positive characteristic Cartier modules, which are related to $D$-modules \cite{Lyubeznik} and are equivalent to \'etale local systems in the appropriate sense \cite{BB11,BB13}, seem to be the correct analog. A Cartier module is simply a triple $(M, \phi, s)$, where $M$ is a coherent sheaf, $s>0$ an integer and $\phi$ is a homomorphism 
$F_*^s M \to M$, for the 
absolute Frobenius morphism $F$ (if $s=1$, we 
omit it from the notation). One example of Cartier modules is $(\omega_X, \phi)$, where $\phi$ is the Grothendieck trace of $F$. In particular, for all integers $e \geq 0$ this yields a composition
\begin{equation}
\label{eq:canonical_Cartier_module}
\xymatrix{\phi ^e:F^e_*\omega _X\ar[r]^{F^{e-1}_*(\phi)} & F^{e-1}_*\omega _X \ar[r]^{F^{e-2}_*(\phi)} & \ldots \ar[r]^{\phi} & \omega _X}.
\end{equation}
From equation \eqref{eq:canonical_Cartier_module} we obtain many important invariants of a positive characteristic variety $X$, that are the fundamental objects of the current paper. The following is a short description the most important of these invariants.
\begin{itemize}
\item For a map $a : X \to A$, applying $a_* ( ? )$ to \eqref{eq:canonical_Cartier_module} yields another Cartier module $(a_*\omega _X,a_*(\phi ))$. The stable image of $a_*(\phi^e)$ (which is the same for all $e \gg 0$) is denoted by $S^0 a_* \omega_X$. The natural replacement for $a_* \omega_X$ in positive characteristic is then either $S^0 a_* \omega_X$ or  $\displaystyle\varprojlim_e  F^e_* S^0a_* \omega_X$ (we will use both depending on the context). 
\item Applying $H^0(X, ? \otimes L)$ for some line bundle $L$ to \eqref{eq:canonical_Cartier_module} does not yield a Cartier module. However, since $H^0(X, \omega_X \otimes L)$ is a finite dimensional vector space (whenever $X$ is projective over a field), the image of $\phi^e$ stabilizes. This stable image is denoted by $S^0(X, \omega_X \otimes L)$. It is a well behaved subset of all the sections $H^0(X, \omega_X \otimes L)$, stable under the Frobenius action. It is also used  to define the Frobenius stable Kodaira dimension $\kappa_S(X)$ which has a similar behaviour to the usual Kodaira dimension. For example $\kappa_S(X)=0$ exactly if $\dim_k S^0(X, \omega_X^m) =1$ for every divisible enough $m\in \mathbb N$ and if $\kappa_S(X)=0$ then $\kappa(X)=0$. 
\item Further, under mild technical assumptions, we can obtain a Cartier module starting from a log pair $(X, \Delta )$ instead of a smooth variety $X$. This is somewhat technical, but crucial for our purposes, see \eqref{lem:S_f_*_mK_X}.
\end{itemize}
One fundamental reason why it is convenient to replace $a_* \omega_X$ by $\Omega:=\displaystyle\varprojlim_e F^e_*S^0a_*  \omega_X$ is that it satisfies a Kodaira-vanishing type result: $H^i(X,\Omega \otimes  N )=0$ for all $i>0$ and any ample line bundle $N$. The proof is an easy combination of Serre-vanishing and the vanishing of cohomology for inverse limits guaranteed by the ML-condition. Using this we show in \eqref{l-ve} that  $H^i(X,\Omega\otimes \hat L ^\vee )=0$ for $i>0$, which is the analogue of (3) of \eqref{t-eqGV}. Note that the proof is somewhat harder than in characteristic zero, because of the presence of inseparable covers. 

At this point, one would be tempted to apply \eqref{t-eqGV} to obtain the other two points of \eqref{t-eqGV}. However, $\Omega$ is typically not coherent and in fact not even quasi-coherent so that  \eqref{t-eqGV} does not apply. Never the less we are  able to prove a Frobenius stable analogue of  (2) of \eqref{t-eqGV}, stating that the stable part of $\sH^i(R \hat S (D_A(F)))$ under the Cartier module action on $F=S^0 a_*\omega_X$ is zero.
\begin{thm}\label{t-m} ( cf.  \eqref{c-0}) Let $X$ be a projective variety over an algebraically closed field $k$ of characteristic $p>0$, $a:X\to A$ a morphism to an abelian variety.
Then for every $i<0$,
$$\varinjlim \sH^i(R \hat S (D_A(F^e_* S^0 a_*\omega_X))) = 0,$$
or equivalently for every integer $e \gg 0$,
$$\im( \sH^i(R \hat S (D_A(S^0 a_*\omega_X))) \to  \sH^i(R \hat S (D_A(F^e_* S^0 a_*\omega_X)))) = 0.$$
\end{thm}

We also remark that it is expected that  $\varprojlim F^e_* R^i a_*\omega_X=0$ for $i>\dim (X/a(X))$. If this is the case, then it is likely that following \eqref{t-m}, one can establish a Frobenius stable version of \eqref{t-dcgv}. Further, we remark that we prove a general version of \eqref{t-m} in \eqref{c-0} pertaining to arbitrary Cartier modules. In particular it applies to the higher direct images $R^j a_* \omega_X$ as well.

We are unable to prove the analog of (1) of \eqref{t-eqGV}. Theorem
\eqref{t-m} does however imply some weak versions of the more traditional generic vanishing statement, see  \eqref{c-gv} where it is shown that   \begin{cor} \label{c-gv-intro} (cf. \eqref{c-gv}) In the situation of \eqref{t-m}, set $\Omega:= \varprojlim F^e_* a_* \omega_X$. Then there exists a proper closed subset $Z\subset \hat A$ such that if $i>0$ and  $y\in V^i(\Omega )=\{ y\in \hat A|h^i(\Omega \otimes P_y)\ne 0 \}$ then $p^m y\in Z$ for all $m\gg 0$.\end{cor}
We also recover a weak version of Simpson's result (see \eqref{p-finite1}).
\begin{thm} \label{p-finite1-intro} (cf. \eqref{p-finite1}) In the situation of \eqref{t-m}, if the reduced Picard variety of $X$ has no supersingular factors (cf. \eqref{ss-F}), then each maximal dimensional component of the closure of the set of points $y\in \hat A$ such that $h^0(\Omega \otimes a^*P_y)\ne 0$ is a finite union of torsion translates of abelian subvarieties, where $\Omega= \varprojlim F^e_* \omega_X$.
\end{thm}
Similar results for $\{y\in \hat A|S^0(X,\omega _X\otimes a^*P_y)\ne 0\}$ and for the support of $\Lambda$ are obtained in \eqref{p-finite} and \eqref{p-lambda}.
It should be noted that we are unable to prove the analog of the inclusions $V^i(\omega _X)\supset V^{i+1}(\omega _X)$ (which holds for m.A.d. projective varieties over $\mathbb C$) nor the result on the reducedness of the loci $V^i(\omega _X)$.

Even though these results seem to be somewhat technical,
it turns out that they have several nice applications which mirror the characteristic $0$ theory.
In particular we prove in  \eqref{t-surj} and \eqref{t-bir} the previously mentioned \eqref{t-one}. Furthermore, we show:

\begin{thm}
(cf. \eqref{thm:closed_subvariety})
If $a : X \hookrightarrow A$ is a closed, smooth subvariety of general type of an abelian variety, then the smallest abelian subvariety $\hat B \subseteq \hat A$ such that the union of finitely many translates of $\hat B$ contains $V^0(A, S^0 a_* \omega_X)$ is equal to $\hat A$.
\end{thm}

\subsection{Organization} In Section \ref{sec:preliminaries} we recall the important facts about derived categories \eqref{ss-derived_categories}, $F$-singularities \eqref{ss-F_singularities}, the Frobenius morphism on abelian varieties \eqref{ss-F}, the Fourier-Mukai transform on abelian varieties \eqref{ss-Fourier_Mukai}, the behaviour of $S^0$ in families \eqref{ss-S_0_families} and higher direct images of the canonical bundle \eqref{ss-direct_images}. In Section \ref{sec:generic_vanishing} we first prove our main generic vanishing theorem \eqref{ss-GVT} and draw some consequences \eqref{ss-consequences_of_GVT}. Then, we prove in subsection \eqref{ss-cohomo_support_loci} the statements \eqref{c-gv-intro} and \eqref{p-finite1-intro} about cohomology support loci, and finally we give some examples \eqref{ss-examples}. In Section \ref{sec:geometry} we define the Frobenius stable Kodaira dimension  
\eqref{ss-K_S} and then we prove Theorem \eqref{t-one} in \eqref{ss-thm_1} and \eqref{ss-thm_2}.

\subsection{Acknowledgments}
The authors would like to thank B. Bhatt and K. Schwede  for useful discussions and comments. The
first named author was partially supported by NSF research grant DMS-0757897 and DMS-1300750 and a grant from the Simons foundation. The second name author would like to thank the  NCTS Mathematics Division (Taipei Office),  where part of  this article was completed during a fruitful visit.

\section{Preliminaries} \label{sec:preliminaries}

We fix an algebraically closed field $k$, of characteristic $p>0$. All schemes will be over $k$ unless otherwise stated. 

\subsection{Derived categories} \label{ss-derived_categories}
Let $X$ be a quasi-compact and separated scheme and $D(X)$ be its derived category (i.e. the derived category of $\mathcal O _X$ modules).  $D_{\rm qc}(X)$ denotes the full subcategory of $D(X)$ consisting of complexes whose cohomologies are  quasi-coherent. For any $F\in D(X)$, $F[n]$ denotes the object obtained by shifting $F$, $n$ places to the left and $\sH^n(F)$ denotes the $\mathcal O_X$ module obtained by taking the $n$-th homology of a complex representing $F$. 
Recall the following.
\begin{thm}[Projection formula]\label{PF} Let $f:X\to Y$ be a morphism of separated, quasi-compact schemes, then there is a functorial isomorphism $$Rf_*(F)\otimes _{\mathcal O _Y}G\to Rf_*(F\otimes _{\mathcal O_X}Lf^*G)$$ for any 
$F\in D_{\rm qc}(X), G\in D_{\rm qc}(Y)$. (Here $\otimes$ is taken in the left-derived sense.)\end{thm}\begin{proof} 
\cite[Proposition 5.3]{Neeman}. \end{proof}
If $X$ is a variety of dimension $n$ over a field $k$ and $\omega ^\cdot_X$ denotes its dualizing complex (which is by definition $f^! \sO_k$ for Hartshorne's $f^!$, such that $\sH^{-\dim X}(\omega^\cdot_X) \cong \omega_X$), then the dualizing functor $D_X$ is defined by $D_X(F)=R{\mathcal H}om(F,\omega _X^\cdot)$ for any $F\in D_{\rm qc}(X)$. We have
\begin{thm}[Grothendieck Duality] Let $f:X\to Y$ be a proper morphism of quasi-projective varieties over a field $k$,  then $$Rf_*D_X(F)=D_YRf_*(F)\qquad \forall\ F\in D_{\rm qc}(X).$$\end{thm}
\begin{proof} For bounded $F$  it is shown in \cite[\S VII]{Hartshorne66}. The general case is in \cite{Neeman}.  
\end{proof}

\begin{defn}
Given a direct system of objects $\sC_i\in D(X)$
\begin{equation*}
\xymatrix{
\sC_1 \ar[r]^{f_1}  & \sC_2 \ar[r]^{f_2} & \dots
}
\end{equation*}
$\hocolim \sC_i$ is defined by the following triangle 
\begin{equation*}
\xymatrix{
\bigoplus \sC_i \ar[r] & \bigoplus \sC_i \ar[r] &  \hocolim \sC_i \ar[r]^>>{+1} &
},
\end{equation*}
where the first map is  the homomorphism given by ${\rm id}-{\rm shift}$,  and
"shift"
denotes the map
$\bigoplus \sC _i\to \bigoplus \sC_i$ defined on $\sC _i$ by the composition $\sC_i\to \sC_{i+1}\subset \bigoplus \sC_j$.
(which is called 1-shift in \cite{Neeman}). 
\end{defn}

\begin{lemma} \label{l-dir} Let $\sC _i \overset{f_i}\to\sC _{i+1}$
be a direct system in $D_{\rm qc}(X)$. Then homotopy colimits commute
with tensor products, pullbacks and pushforwards. In particular we have \begin{enumerate}
\item $\hocolim \sH^j( \sC _i )=\sH^j( \hocolim \sC_i )$, and
\item $\hocolim  R^j\Gamma  (\sC _i )=R^j\Gamma ( \hocolim \sC _i)$.
\end{enumerate}
\end{lemma}\begin{proof} See \cite[Lemma 2.8]{Neeman} (with $c= \sO_X$ and $\sT= D_{\rm qc}(X)$) or \cite[2.11]{Kuznetsov11}. The fundamental reason is that for any additive functor $F$ from $D_{\rm qc}(X)$ to the category of abelian groups, $F({\rm id}-{\rm shift})$ is injective.  Indeed, let $(c_i) \in \oplus F(\sC_i)$, and let $(d_i) \in \oplus F(\sC_i)$ be the image via $F({\rm id}-{\rm shift})$. That is, 
\begin{equation*}
d_1 = c_1 \qquad d_2 = c_2 - F(f_1)(c_1) \qquad \dots \qquad d_i = c_i - F(f_{i-1})(c_{i-1}) \qquad \dots
\end{equation*}
Then, we have
\begin{equation*}
c_1 = d_1 \qquad c_2 = d_2 - F(f_1)(d_1) \qquad c_3 = d_3 - F(f_2)(d_2 - F(f_1)(d_1)) \quad \dots 
\end{equation*}
In particular, if all $d_i$ are zero, then so are all the $c_i$. 
\end{proof}

\begin{defn}
  Given an inverse system of objects $\sC_i\in D_{\rm qc}(X)$
\begin{equation*}
\xymatrix{
\sC_1 & \ar[l]^{\tilde{f}_1} \sC_2 & \ar[l]^{\tilde{f}_2}  \dots
}
\end{equation*}
then $\holim \sC_i$ is defined by the following triangle \cite[Definition 29]{Murfet}
\begin{equation*}
\xymatrix{
 \holim \sC_i \ar[r] & \prod \sC_i \ar[r] &  \prod \sC_i \ar[r]^{+1} &.
}
\end{equation*}
Here the map between products is $\prod  \left( \id - {\rm shift} \right)$. Note also that by product we mean the ordinary product of chain complexes (which is well-defined on the derived category and is the product in $D(X)$),  not the product inside $D_{\rm qc}(X)$. In particular then $\holim C_i$ is an object of $D(X)$, not of $D_{\rm qc}(X)$.
\end{defn}

It is easy to check that if $\sC_i$ are coherent sheaves, then $\hocolim \sC_i = \varinjlim \sC_i$ (see the proof of \eqref{l-dir}).  However, $\holim \sC_i \neq \varprojlim \sC_i$ in general. For an easy example, consider $\sC_i:= \sO_{\mathbb{A}^1_k}$, where $\mathbb{A}_1 =\Spec k[x]$, with the maps $\sC_{i+1} \to \sC_i$ being multiplication by $x$. Then $\prod \sC_i \to \prod \sC_i$ is not surjective, since $(1) \in \prod \sO_{\mathbb{A}^1_k}$ is not in the image. 

\begin{lemma}
\label{lem:holim_equals_lim}
Given  an inverse system of quasi-coherent sheaves $\sC_i$ 
\begin{equation*}
\xymatrix{
\sC_1 & \ar[l]^{\tilde{f}_1} \sC_2 & \ar[l]^{\tilde{f}_2}  \dots
}
\end{equation*}
satisfying the ML-condition, that is, for every $i$,  $\im (\sC_j \to \sC_i)$ is the same for all $j \gg i$, we have
\begin{equation*}
\holim \sC_i = \varprojlim \sC_i.
\end{equation*}
\end{lemma}

\begin{proof}
It is immediate that 
\begin{equation*}
\varprojlim \sC_i = \ker \left( \prod \sC_i \to \prod \sC_i \right).
\end{equation*}
Hence, to prove the required equality we have to show that $\prod \sC_i \to \prod \sC_i$ is surjective. For that it is enough to prove that $\prod \sC_i(U) \to \prod \sC_i(U)$ is surjective for each affine open set $U$. However, there the question becomes a question on abelian groups, which is well known (see for example \cite[\href{http://stacks.math.columbia.edu/tag/0123}{Tag 07KW}, (3)]{stacks-project} and  \cite[ \href{http://stacks.math.columbia.edu/tag/0123}{Tag 0594}]{stacks-project} for the definition of the ML-condition).
\end{proof}

Further, using the language of \cite[Lemma 2.8]{Neeman} one obtains the following.

\begin{lemma}
\label{lem:direct_inverse_limit_alternative}
If 
\begin{equation*}
\xymatrix{
\sC_1 \ar[r]^{f_1}  & \sC_2 \ar[r]^{f_2} & \dots
}
\end{equation*}
is a direct system in $D_{\rm qc}(X)$ and $\sD \in D_{\rm qc}(X)$, then
\begin{equation*}
\mathcal{RH}om ( \hocolim \sC_i , \sD) \cong \holim \mathcal{RH}om(  \sC_i, \sD) . 
\end{equation*}
\end{lemma}

\begin{proof}
Apply $\mathcal{RH}om ( \_ ,\sD)$ to the triangle
\begin{equation*}
\bigoplus \sC_i \to \bigoplus \sC_i \to \hocolim \sC_i \stackrel {+1}\to .
\end{equation*}
We obtain the triangle
\begin{equation*}
\mathcal{RH}om ( \hocolim \sC_i, \sD) \to \mathcal{RH}om ( \bigoplus \sC_i, \sD) \to \mathcal{RH}om (\bigoplus \sC_i, \sD)  \to^{+1}.
\end{equation*}
Notice now that 
\begin{equation}
\label{eq:direct_inverse_limit_alternative:prod_co_prod}
\mathcal{RH}om (\bigoplus \sC_i, \sD) \cong \prod \mathcal{RH}om (\sC_i, \sD). 
\end{equation}
Indeed, for every $i$ there is a natural map $\mathcal{RH}om (\bigoplus \sC_i, \sD) \to \mathcal{RH}om (\sC_i, \sD) $. This induces a natural map  $ \mathcal{RH}om (\bigoplus \sC_i, \sD) \to \prod \mathcal{RH}om (\sC_i, \sD) $. To prove that it is an isomorphism, it is enough to prove that it induces an isomorphism on each cohomology sheaf. By restricting to affine patches we may also replace $\mathcal{RH}om$ by $\mathcal{R}\mathrm{Hom}$. That is we have to show that $\prod \mathrm{Hom}_{\sD(X)} (\sC_i, \sD) \to \mathrm{Hom}_{\sD(X)} (\bigoplus \sC_i, \sD)  $ is an isomorphism. This holds because of the universal property of $\bigoplus$. This concludes the proof of \eqref{eq:direct_inverse_limit_alternative:prod_co_prod}.
 
Then the previous triangle translates to 
\begin{equation*}
\mathcal{RH}om ( \hocolim \sC_i, \sD) \to \prod \mathcal{RH}om (\sC_i, \sD) \to \prod \mathcal{RH}om (\sC_i, \sD) \to^{+1}
\end{equation*}
and the map between the products is just $\prod \left( \id_{\mathcal{RH}om (\sC_i, \sD)} - \mathcal{RH}om (f_{i-1}, \sD)\right) $. This shows that $\mathcal{RH}om ( \hocolim \sC_i, \sD)$ is indeed the homotopy limit of the inverse system
\begin{equation*}
\xymatrixcolsep{80pt}
\xymatrix{
\mathcal{RH}om ( \sC_1, \sD) &  \ar[l]^{\mathcal{RH}om ( f_1, \sD) } \mathcal{RH}om ( \sC_2, \sD) &  \ar[l]^{\mathcal{RH}om ( f_2, \sD) } \dots
}
\end{equation*}
\end{proof}

\subsection{F-singularities}\label{ss-F_singularities}
The title is somewhat misleading.  We define here the global invariants whose origin lies in the theory of $F$-singularities. For the general theory we refer to \cite{ST} and \cite{S11}. Throughout this subsection the letter $F$ stands for the absolute Frobenius morphism of the given variety. 

\begin{defn} Let $X$ be a smooth, proper variety over $k$, $\Delta \geq 0$ a $\bQ$-divisor, $s >0$ an integer, such that $(p^s -1)\Delta$ is an integral divisor, $f : X \to Y$ a morphism over $k$ and $M$ a Cartier divisor on $X$. We define the subsheaf $S^0f_*(\sigma (X,\Delta ) \otimes \mathcal O_X(M)) \subseteq f_* \sO_X(M)$ to be the intersection:
$$\bigcap _{e\geq 0}{\rm Image}\left( {\rm Tr}^{es}F^{es}_*f_*\mathcal O _X((1-p^{es})(K_X +\Delta )+p^{es}M)\to f_*\mathcal O_X (M)\right), $$
where ${\rm Tr}^{es}$ is obtained from the Grothendieck trace $F^{es}_* \omega_X \to \omega_X$ of $F^{es}$ by twisting with $\sO_X(M - K_X)$, pushing forward by $f$ and applying that $F_* f_* = f_* F_*$. In the special case of $Y = \Spec k$, we use the notation $S^0(X,\sigma(X, \Delta) \otimes \sO_X(M))$ instead of $S^0f_*(\sigma (X,\Delta ) \otimes \mathcal O_X(M))$.
\end{defn}
This intersection is a descending intersection, so a priori it needs not stabilize. In this case $S^0f_*(\sigma (X,\Delta ) \otimes \mathcal O_X(M))$ tends not to be coherent. There are several cases when the intersection stabilizes, for example if $M-K_X-\Delta$ is ample (cf. \cite[2.15]{HX13}).

\begin{lem}
\label{lem:obvious_inclusion}
With the above notation, there is a natural inclusion $S^0(X,\sigma (X,\Delta ) \otimes \mathcal O_X(M)) \subseteq H^0(Y, S^0f_*(\sigma (X,\Delta ) \otimes \mathcal O_X(M)))$.
\end{lem}

\begin{proof}
Denote 
\begin{equation*}
\sF_e := F^{es}_*f_*\mathcal O _X((1-p^{es})(K_X +\Delta )+p^{es}M) .
\end{equation*}
Then, 
\begin{equation*}
S^0(X,\sigma (X,\Delta ) \otimes \mathcal O_X(M)) = \bigcap_{e \geq0} \im (H^0(Y,\sF_e) \to H^0(Y,\sF_0) ),
\end{equation*}
while
\begin{equation*}
 H^0(Y, S^0f_*(\sigma (X,\Delta ) \otimes \mathcal O_X(M))) = H^0 \left(Y, \bigcap_{e \geq0} \im (\sF_e \to \sF_0 ) \right).
\end{equation*}
The inclusion then follows from the following computation.
\begin{multline*}
 \bigcap_{e \geq0} \im (H^0(Y,\sF_e) \to H^0(Y,\sF_0) ) \subseteq \bigcap_{e \geq0} H^0 \left( Y, \im (\sF_e \to \sF_0 ) \right) \\ = H^0 \left( Y, \bigcap_{e \geq0}  \im (\sF_e \to \sF_0 ) \right).
\end{multline*}

\end{proof}

\begin{lem} \cite[Lemma 2.6]{Pat13}
\label{lem:S_f_*_mK_X}
Let $X$ be a smooth variety and  $D \in |mK_X|$ for some integer $m>0$ coprime to $p$. Assume further that $f : X \to Y$ is a proper morphism over $k$. Define $s>0$ to be the smallest integer such that $m |(p^s -1)$ and $\Delta:= \frac{m-1}{m}D$.  Then the chain 
\begin{multline}
\label{eq:S_f_*_mK_X:one_step}
\dots \to f_* F^{(e+1)s}_* \sO_X ( m p^{(e+1)s} K_X + (1-p^{(e+1)s}) (K_X + \Delta) )  \to  \\
 \to f_* F^{es}_* \sO_X ( m p^{es} K_X + (1-p^{es}) (K_X + \Delta) ) \to \dots 
\end{multline}
is isomorphic to 
\begin{equation}
\label{eq:S_f_*_mK_X:stream}
\dots \to f_* F^{(e+1)s}_* \sO_X(mK_X) \xrightarrow{f_* F^{es}_*(\alpha)} f_* F^{es}_* \sO_X(mK_X) \xrightarrow{f_* F^{(e-1)s}_*(\alpha)}  f_* F^{(e-1)s}_* \sO_X(mK_X) \to \dots,
\end{equation}
where $\alpha$ is the usual homomorphism induced by the Grothendieck trace of Frobenius
\begin{equation*}
F^s_* \sO_X(mK_X) \cong F^s_* \sO_X ( m p^s K_X+(1-p^s)(K_X + \Delta)  )  \to  \sO_X(m  K_X).
\end{equation*}
In particular, the intersection in the definition of $\Omega_0 := S^0f_*(\sigma(X,\Delta) \otimes \sO_X(mK_X))$ stabilizes by \cite[Proposition 8.1.4]{BS12} and agrees with the  image of 
$$f_* F^{(e+1)s}_* \sO_X(mK_X) \to f_* \sO_X(mK_X)$$
for $e \gg 0$. Furthermore, these then form a chain of surjective maps
\begin{equation*}
\dots \twoheadrightarrow F^{2s}_* \Omega_0 \twoheadrightarrow F^s_* \Omega_0 \twoheadrightarrow \Omega_0 .
\end{equation*}

\end{lem}


\begin{lem}
\label{lem:slightly_less_obvious_containment}
If $a : X \to A$ is a  proper morphism  from a smooth variety $X$ to a scheme $A$ over $k$, then $S^0(X, \omega_X) = S^0(A, S^0 a_* \omega_X)$.
\end{lem}

\begin{proof}
Suppose that $f\in S^0(A, S^0 a_* \omega_X)$, then there are elements $$f_e\in H^0(A, F^e_*S^0 a_* \omega_X)\subset H^0(A,  F^e_*a_* \omega_X)\cong H^0(X,   F^e_*\omega_X)$$
such that $H^0(A,a_*({\rm tr}_{F^e}))(f_e)=f$ for all $e\geq 0$.
We write $\tilde f$ and $\tilde f_e$ for the corresponding elements in $H^0(X,\omega _X)$ and $H^0(X,   F^e_*\omega_X)$. Since the following diagram commutes
\begin{equation*}
\xymatrix{
H^0(A, F^e_*S^0 a_* \omega_X)\ar[r] \ar@{^(->}[d] & H^0(A, S^0 a_* \omega_X)\ar@{^(->}[d] \\
H^0(A, F^e_*a_* \omega_X)\ar[r] \ar@{=}[d] & H^0(A, a_* \omega_X)\ar@{=}[d] \\
H^0(X, F^e_* \omega_X)\ar[r] & H^0(X,  \omega_X)
},
\end{equation*}
it follows that $\tilde f= H^0(X,{\rm tr}_{F^e})(\tilde f_e)$ so that $\tilde f\in S^0(X,   \omega_X)$. Thus $S^0(X, \omega_X) \supset S^0(A, S^0 a_* \omega_X)$.

For the reverse inclusion, note that if $\phi : F_* \omega_X \to \omega_X$ is the Grothendieck trace, then  $(a_* \omega_X, a_*( \phi))$ is a Cartier module. Hence, by \cite[Proposition 8.1.4]{BS12}, there is an $e$, such that $F^e_* a_* \omega_X \to S^0 a_* \omega_X$ is surjective. Let then $e' \geq 0$ be arbitrary. The statement of the lemma follows from the  commutative diagram below (commutativity up-to-unit follows from \cite[Lemma 3.9]{S09}).
\begin{small}\begin{equation*}
\xymatrix@C=50pt{
H^0(X, F^{e+e'} _*\omega_X) = H^0(A,a_* F^{e + e'} _*\omega_X ) \ar[rdd]_{H^0\left(X,{\rm tr}_{F^{e+e'}} \right)} \ar[r]^-{H^0(A,F^{e'}_*a_*({\rm tr}_{F^{e}}))} & H^0(A,F^{e'}_* S^0 a_* \omega_X ) \ar[d]_{H^0 \left(A,a_*\left({\rm tr}_{F^{e'}}\right) \right)} \\
 & H^0(A, S^0 a_* \omega_X) \ar@{^(->}[d] \\
 & H^0(A, a_* \omega_X) = H^0(X, \omega_X)
}
\end{equation*}\end{small}
\end{proof}

\subsection{The Frobenius morphism on abelian varieties}\label{ss-F}
Throughout this paper $A$ will denote an abelian variety of dimension $g$ over $k$, $\hat A={\rm Pic }^0(A)$ the dual abelian variety and $\sL$ the normalized Poincar\'e line bundle on $A\times \hat A$. Further $V_A$ and $V_{\hat A}$ are the Verschiebung isogenies \cite[5.18]{MvdG}, where we drop the subindex whenever it is clear from the context.

Given a scheme $X$ over $k$ with structure morphism $\mu : X \to \Spec k$, we denote by $X'$ the twisted version of $X$, i.e., the scheme over $k$ which is identical to  $X$ as an abstract scheme but its $k$-structure is $F_k \circ \mu $. If instead of $F_k$ we compose with $F_k^m$ then we write $X^{(m)}$, where $m$ can be negative as well. The main reason for introducing $X'$, and $X^{(m)}$ in general, is that the absolute Frobenius morphism $F : X \to X$ is not a $k$-morphism. To treat it  as a $k$-morphism one has to regard it as a morphism $X' \to X$. For many purposes (e.g., the content of Subsection \ref{ss-F_singularities}), this is not necessary and taking the domain of the Frobenius to be $X$ as well is more convenient (e.g., for defining iterations of a Frobenius action). However, when the $k$-structure is important, e.g., when one takes the product of $X$ with another $k$-variety (which we will do frequently), then it is important to treat the domain of the Frobenius as $X'$. Further, our notions 
will be based on dimensions of $k$-vector spaces, which is invariant, by the perfectness assumption on $k$, under twisting the $k$-structure by the Frobenius morphism of $k$. Hence we will freely change between the two point of views, according to which is more adequate, sometimes leaving a few details to the reader.

\begin{lem}
\label{lem:dual_of_Frobenius}
Given an abelian variety $A$ over $k$, the dual of $A'$ is isomorphic to   $ (\hat A )'$ (hence we denote both by $\hat A  '$). Further  the Verschiebung $ V  : \hat A \to \hat A '$ is the dual of the Frobenius $F : A' \to A$. 
\end{lem}

The proof is an straight forward application of the Seesaw principle.

\begin{proposition}\label{p-ordinary} Let $A$ be an abelian variety over $k$ of dimension $n$.
The following are equivalent.
\begin{enumerate}
\item $A$ is ordinary in the sense of \cite[7.2]{BK},
\item there are $p^g$ $p$-torsion points,
\item $V$ is \'etale,
\item the Frobenius action  $H^n(A, \mathcal O_A) \to H^n(A, \mathcal O_A)$ is bijective, 
\item the Frobenius action  $H^i(A, \mathcal O_A) \to H^i(A, \mathcal O_A)$ is bijective for $0\leq i\leq n$,
\item $A$ is globally F-split,
\item $S^0(A,\omega _A) \neq 0$. 
\end{enumerate}
\end{proposition}

Note that apart from point $(2)$, the rest is equivalent also over non algebraically closed perfect fields, because they are properties that are invariant under base-field extension. 

\begin{proof} (1) and (2) are equivalent by \cite[7.4]{BK}. 

(2) and (3) are equivalent 
because as $V:\hat A\to \hat A$ is an isogeny, $V$ is the quotient by the scheme theoretic inverse image $V^{-1}(0)$. Hence $V$ is \'etale if and only if $V^{-1}(0)$ is reduced. However $V^{-1}(0)$ has always length $p^g$ (because $\deg V = p^g$), and so $V^{-1}(0)$ is reduced if and only if it contains $p^g$ points. Since the points of $V^{-1}(0)$ are exactly the $p$-torsion points, we are done.

(2) and (4) are equivalent by \cite[5.4]{MS11}.

(4) and (5) are equivalent by \cite[5.4]{MS11}.

(4) and (7) are equivalent because the morphism $\mathcal O_A \to F_* \mathcal O_A$, which induces the map of (4) is dual to the Grothendieck trace. So, dualizing the map of (4) one obtains $H^0(A, F_* \omega_A) \to  H^0(A, \omega_A)$.

(6) and (7) are well known to be equivalent.  
\end{proof}
If any of the above equivalent conditions holds, we say that $A$ is {\it ordinary}. We have the following easy lemma.
\begin{lemma} \label{l-isogeny} Let $\varphi:A\to B$ be an isogeny between abelian varieties of dimension $n$, then $A$ is ordinary if and only if $B$ is ordinary. In particular $A$ is ordinary if and only if $\hat A$ is ordinary.\end{lemma}
\begin{proof} This is immediate by (2) of \eqref{p-ordinary} and \cite[5.22]{MvdG}. For the addendum, note that any abelian variety is isogenous to its dual abelian variety.
\end{proof}
\begin{lemma}\label{l-sq} Let $\varphi:A\to B$ be a surjective morphism of abelian varieties. If $A$ is ordinary, then so is $B$.\end{lemma}
\begin{proof} Let $d=\dim B$. By \eqref{l-isogeny} and \cite[p. 173]{Mumford} we may assume that $A \to B$ is a projection onto a factor.  By (5) of \eqref{p-ordinary}, $F_A^*:H^d(A, \mathcal O_A) \to H^d(A, \mathcal O_A)$ is bijective. Since $\varphi : H^d(B, \mathcal O_B)\to H^d(A, \mathcal O_A)$ is injective and  $F_B\circ \varphi =\varphi\circ F_A$, it follows that  $F_B^*:H^d(B, \mathcal O_B) \to H^d(B, \mathcal O_B)$ is bijective and hence that $B$ is ordinary.
\end{proof}
 Following \cite{PR03} we say that
 an abelian variety $A$ endowed with an isogeny $\varphi:A\to A$ is {\it pure of positive weight} if there exists $r,s>0$ such that $\varphi ^s={F}_{p^r}$ for some model of $A$ over $\mathbb F_{p^r}$. In particular $\varphi$ is purely inseperable. 
 If $A$ is defined over a finite field, then $A$ is {\it supersingular} if and only if it is pure of positive weight for the isogeny given by multiplication by $p$, in general $A$ is supersingular if and only if it is isogenous to a supersingular abelian variety defined over a finite field. 
Further, as in \cite{PR03}, $A$ has 
 {\it no supersingular factor} if there does not exists a nontrivial homomorphism to a supersingular  abelian variety. 

Since for an ordinary abelian variety $[p]= V \circ F$ is never purely inseparable by \eqref{p-ordinary}, an ordinary abelian variety is never pure of positive weight for the isogeny given by multipliation by $p$. So, we have the following.
\begin{lemma} If $A$ is an ordinary abelian variety, then $A$ has no supersingular factors.\end{lemma}
\begin{proof} Immediate from \eqref{l-sq}.
\end{proof}
\begin{thm}\label{t-pr} Let $A$ be an abelian variety over $k$ and $ X \subset  A$ a reduced closed subscheme satisfying $p(X) = X$ (resp. $p(X)\subset X$). If $A$ has no supersingular factors, 
 then $X$ is completely linear i.e. a finite union of torsion translates of sub abelian varieties (resp. all maximal dimensional irreducible components of $X$ are completely linear).
\end{thm}
\begin{proof} See \cite[2.2]{PR03} and the proof of  \cite[4.1]{PR03}.\end{proof}
 
\subsection{Abelian varieties and the Fourier-Mukai functor} \label{ss-Fourier_Mukai}
The Fourier-Mukai transforms $R\hat S:D(A)\to D(\hat A)$ and $R S:D(\hat A)\to D( A)$ are defined by
$$R\hat S (?)=Rp_{\hat A,*}(Lp_A^*?\otimes \sL ), \qquad R S (?)=Rp_{A,*}(Lp_{\hat A}^*?\otimes \sL ).$$
Note that $p_A^*$ and $p_{\hat A}^*$ are exact so, sometimes $L$ is omitted in front of them. By \cite{Mukai81}, it is known that:
\begin{thm}\label{M2.2} The following equalities hold on $D_{\rm qc}(A)$ and $D_{\rm qc}( \hat A)$.
$$RS\circ R\hat S=(-1_A)^*[-g],\qquad {\rm and}\qquad R\hat S\circ RS=(-1_{\hat A})^*[-g],$$
where $[-g]$ denotes the shift by $g$ places to the right and $-1_A$ is the inverse on $A$.\end{thm}The following two properties are proven in \cite[3.1 and 3.8]{Mukai81}.
\begin{lem}\label{M3.1} (Exchange of translation and $\otimes {\rm Pic }^0$.) Let $x\in \hat A$ and $P_{x}= \sL |_{A\times x}$, then the following equalities hold on $D_{\rm qc}(A)$.
$$RS \circ T_x^*\cong (\otimes P_{-x})\circ RS.$$\end{lem}
\begin{lem}\label{M3.8} We have $D_A\circ RS\cong ((-1_A)^*\circ RS \circ D_{\hat A})[g]$ on $D_{\rm qc}( \hat A)$.
\end{lem}

Note that in the statement of the following lemma we consider the functors on the whole $D(A)$ not only the subcategory $D_{\rm qc}( \hat A)$.

\begin{lem}
\label{lem:D_A_T_x}
For every $x \in A$, $T_x^* \circ D_A   \cong D_A \circ T_x^*$ as functors on $D(A)$. 
\end{lem}

\begin{proof}
By the proof of \cite[Proposition III.2.2]{Hartshorne77} one can choose an injective resolution $\sI$ of $\omega_A^\bullet$, for which $T_x^* \sI \cong \sI$ (as complexes, not just a quasi-isomorphism). Then one can compute $T_x^* \circ D_A(\sD)$ as $T_x^* (\sH om^\bullet( \sD,\sI))$ and $D_A \circ T_x^*(\sD)$ as $\sH om^\bullet(T_x^* \sD, \sI)$. Note here we have used that $T_x$ is an automorphism, hence $T_x^*$ is exact. Further note that $T_x^* (\sH om^\bullet( \sD,\sI)) \cong  \sH om^\bullet( T_x^* \sD, T_x^*\sI)$ (by checking it on every open set),  hence it suffices to prove that $\sH om^\bullet( T_x^* \sD, T_x^*\sI)$  and  $\sH om^\bullet( T_x^* \sD,\sI)$ are  quasi-isomorphic. However, they are even isomorphic by the choice of $\sI$.
\end{proof}



The following result on the exchange between direct and inverse images is contained in \cite[3.4]{Mukai81}.
\begin{lemma}\label{m3.4} Let $\phi :A\to B$ be an isogeny of abelian varieties and $\hat \phi :\hat B\to \hat A$ the dual isogeny, then the following equalities hold on $D_{\rm qc}(B)$ and $D_{\rm qc}(A)$.
$$\phi ^*\circ RS_B \cong RS_A\circ \hat \phi _* $$
$$\phi _* \circ RS _A \cong RS_B\circ \hat \phi ^* .$$ 
\end{lemma} 
\begin{proof} See \cite[3.4]{Mukai81}.  
\end{proof}
\begin{lemma} \label{l-dirRS} Let $\Lambda _{e}\to\Lambda _{e+1}$  be a direct system in $D_{\rm qc}(A)$. Then $\hocolim R\hat S( \Lambda _e)=R\hat S( \hocolim \Lambda _e)$.
\end{lemma} \begin{proof} We have $$R\hat S( \hocolim \Lambda _e)=Rp_{\hat A,*}(L{p_A^*}(\hocolim \Lambda _e)\otimes \sL)=$$
$$\hocolim Rp_{\hat A,*}(L{p_A^*} \Lambda _e\otimes \sL)=\hocolim R\hat S(  \Lambda _e),$$ 
where the first and last equalities follow from the definition of $R\hat S$ and the middle equality follows by \eqref{l-dir}.
\end{proof}


\subsection{$S^0(\omega_X)$ in families} \label{ss-S_0_families}

By the following theorem from \cite[Theorem 3.3]{Pat13}, the Frobenius stable subspace $S^0 \subseteq H^0$ behaves well in families. We cite here only the special case used in this paper. For the general case please see \cite{Pat13}.

\begin{theorem} \cite[Theorem 3.3]{Pat13}
\label{thm:openness_S_0}
Let $f : X \to Y$ be a proper, surjective, generically smooth morphism of smooth varieties over $k$.   Then, there is a non-empty Zariski open set $W$ of $Y$ such that $S^0(F,\omega_F)$ has the same dimension for every geometric fiber $F$ over $W$. Further, the rank of $S^0 f_* \omega_{X}$ is at least as big as this general value.
\end{theorem}


\subsection{Higher direct images of $\omega _X$} \label{ss-direct_images}
\begin{proposition}\label{p-high} Let $f:X\to A$ be a generically finite dominant morphism of  projective varieties. Assume that $X$ is smooth of dimension $n$. Then ${\rm codim}R^if_*\omega _X\geq i+2$ for all $i>0$. 
\end{proposition}
\begin{proof} We proceed by induction on the dimension.
Let $H$ be a very general sufficiently ample divisor. Pushing forward the short exact sequence
$$0\to \omega _X \to \omega _X (H)\to \omega _H\to 0,$$ one sees that it is enough to prove that 
  ${\rm codim}R^1f_*\omega _X\geq 3$. This can be checked by localizing at a codimension $2$ point, in which case it is a consequence of the relative Kawamata-Viehweg vanishing (which holds for two dimensional  excellent schemes see \cite[2.2.5]{KK94}).
\end{proof}

\section{Generic Vanishing} \label{sec:generic_vanishing}

\subsection{Proof of Theorem \ref{t-m}.} \label{ss-GVT}

Recall that in our notation $k$ is an algebraically closed field of characteristic $p>0$ and $A$   an abelian variety defined over $k$.

\begin{thm}\label{GVT} Let $\psi_{e}:\Omega _{e+1}\to\Omega _e$   be an inverse system of coherent sheaves    
on an abelian variety $A$ such that   for any sufficiently ample line bundle $L$ on $\hat A$ and any $e\gg 0$, $H^i(A,\Omega _e \otimes \hat L^\vee )=0$ for all $i>0$. Then, the complex
\begin{equation*}
\Lambda:=\hocolim R\hat S (D_A(\Omega_e ))
\end{equation*}
is a quasi-coherent sheaf in degree $0$, i.e., $\Lambda =\mathcal H ^0(\Lambda )$. Furthermore, if there is an integer $r>0$, such that the image $\Omega_{e'} \to \Omega_e$ is the same for every $e' \geq e + r$, then
$$\Omega :=\varprojlim \Omega_e = ((-1_A)^* D_A RS (\Lambda ) )[-g].$$
\end{thm} 
\begin{proof} The object $\Omega _e$ 
lives in degree 0, hence $D_A(\Omega _e)$ 
lives in degrees $[-g,\ldots ,0]$ and we have the support condition $${\rm codim } \   {\rm Supp}(\sH^j(D_A(\Omega_e )))\geq g+j.$$
We then have that $\Lambda_e :=R\hat S (D_A(\Omega_e ))$ lives in degrees $[-g,\ldots ,0]$. Define as in the statement $\Lambda:= \hocolim \Lambda_e$. By \eqref{l-dir}, $\Lambda$ also  lives in degrees $[-g,\ldots ,0]$ and it has quasi-coherent cohomologies.

To show the statement about $\Lambda$ we must show that \emph{$\Lambda$ lives in cohomological degree $0$}, i.e., that $\sH^j( \Lambda)=0$ for $j\in [-g,\ldots , -1]$. To see this we argue as follows:
Let $j\in  [-g,\ldots , 0]$ be the smallest integer such that $\sH^j(\Lambda)\ne 0$
and assume that $j\leq -1 $. We have that $\sH^j( \Lambda)=\hocolim \sH^j(\Lambda_e)$ (cf.  \eqref{l-dir}) and so we may fix $e>0$ such that the image of 
$\sH^j(\Lambda_e)\to \sH^j(\Lambda)$ is non-zero.
We twist by  a sufficiently ample line bundle $L$ so that  $\sH^j(\Lambda_e)\otimes L$ is globally generated and hence the image of $$R^0\Gamma (\sH^j(\Lambda_e)\otimes L)\to R^0\Gamma (\sH^j( \Lambda)\otimes L)$$ is non-zero. Let $$E^{i,l}_2=R^i\Gamma (\sH^l(\Lambda)\otimes L)$$ abutting to $R^{i+l}\Gamma (\Lambda \otimes L)$. By our choice of $j$ we have that $E^{i,l}_2=0$ for $l<j$ and hence that 
\begin{equation}
\label{eq:generic_vanishing_cohomology_not_zero}
R^{j}\Gamma ( \Lambda \otimes L) \cong R^{0}\Gamma (\sH ^j(\Lambda)\otimes L)=E^{0,j}_2\ne 0. 
\end{equation}
On the other hand, following the beginning of the proof of \cite[1.2]{Hacon04}, we have that
\begin{multline*}D_k(R\Gamma (\Omega _e \otimes \hat L^\vee ))\cong ^{\rm G.D.}R\Gamma (D_A(\Omega _e \otimes \hat L^\vee )) \cong \\ R\Gamma (D_A(\Omega _e )\otimes \hat L) \cong 
R\Gamma (D_A(\Omega _e )\otimes p_{A,*}(\sL \otimes p_{\hat A}^* L)) \cong ^{\rm P.F.}\\ R\Gamma (L {p_A^*}D_A(\Omega _e )\otimes \sL \otimes p_{\hat A}^* L)) \cong ^{\rm P.F.}R\Gamma (R\hat S D_A(\Omega _e )\otimes L) = R\Gamma (\Lambda_e \otimes L).
\end{multline*}
Since, by assumption, $R^l\Gamma (\Omega _e \otimes \hat L^\vee )=0$ for any $e\gg 0$ and $l>0$, it follows that $R^{-l}\Gamma (\Lambda_e\otimes L)=0$ for any $e\gg 0$
and hence  that $R^{j}\Gamma (\Lambda_e\otimes L)=0$ for any $e \gg 0$. Thus $$R^{j}\Gamma (\Lambda \otimes L)=\varinjlim R^{j}\Gamma (\Lambda_e\otimes L)=0$$ (cf. \eqref{l-dir}). This contradicts \eqref{eq:generic_vanishing_cohomology_not_zero} and hence concludes the first part of our statement. 

The second part is shown by the following stream of isomorphisms.
\begin{multline*}D_{A} RS (\hocolim \Lambda_e) 
\cong \underbrace{D_{ A} (\hocolim RS R\hat S D_A (\Omega _e))}_{\textrm{\eqref{l-dirRS}}}  
\\ \cong \underbrace{\holim D_{ A}( (-1_A)^*  D_A (\Omega _e) [ -g])}_{\textrm{\eqref{lem:direct_inverse_limit_alternative}+\eqref{M2.2}}} 
 \cong \underbrace{( \holim D_{ A} (-1_A)_*  D_A (\Omega _e)) [ g] }_{\textrm{$(-1_A)_* = (-1_A)^*$, because $-1_A \circ -1_A = \Id_A$}}
\\ \cong \underbrace{(-1_A)_* (\holim D_{ A}   D_A (\Omega _e)) [ g]}_{\textrm{$(-1_A)_* D_A = D_A (-1_A)_* $ by G.D.}}   
\cong \underbrace{(-1_A)^* \holim \Omega _e [ g]}_{\textrm{$D_AD_A (\sF)= \sF$ for $\sF$ coherent}} \cong  \underbrace{(-1_A)^* \Omega [g]}_{\textrm{\eqref{lem:holim_equals_lim}}} .\end{multline*}
\end{proof}
 Choose a coherent sheaf $\Omega_0$ on $A$ and an integer $s >0$, such that there is a  homomorphism $\alpha : \Omega_1:=F_*^s \Omega_0 \rightarrow \Omega_0$ (i.e., $(\Omega_0, \alpha,s)$ is a Cartier module, c.f., Subsection \ref{subsec:introduction:GVP}). Then this induces for every integer $e \geq 1$ homomorphisms $$F^{es}_* (\alpha) : \Omega_{e+1}:=F^{(e+1)s}_* \Omega_0 \rightarrow \Omega_e:= F^{es}_* \Omega_0.$$  An example of such setup is that of \eqref{lem:S_f_*_mK_X}, by 
setting  $Y:=A$ and defining $\Omega_0:=S^0 f_* (\sigma(X, \Delta) \otimes \sO_X(mK_X))$.
Notice that for any Cartier module $(\Omega_0, \alpha,s)$, there exists an integer $e_0$ and a coherent subsheaf $\Omega '_0\subset \Omega_0$ (\cite[Lemma 13.1]{Gabber}  \cite[Proposition 8.1.4]{BS12}) such that $$\Omega '_0={\rm Im}\left( \Omega _e\to \Omega _0\right) .$$ We then have that $\Omega \cong \varprojlim F^{es}_*\Omega '_0$. Thus replacing $\Omega _0$ by $\Omega '_0$ we may assume that the homomorphisms $\Omega _{e+1}\to \Omega _e$ are surjective.

\begin{lemma}\label{l-ve} With notation as above, let $L$ be an ample line bundle on $\hat A$ and $\hat L =RS(L)=R^0S(L)$ be its Fourier Mukai transform. If $\Omega _e =F^{es}_*\Omega _0$,
then $H^i(A, \Omega _e \otimes \hat L ^\vee \otimes P)=0$ for $e\gg 0$, any $i>0$ and $P\in \hat A$.
\end{lemma}
\begin{proof} Recall that $\phi _L^*(\hat L ^\vee )\cong L^{\oplus {h^0(L)}}$. 
Note that \begin{multline*}H^i(A, \Omega _e \otimes \hat L ^\vee \otimes P)\cong H^i(A, F^{es}_{A,*} \Omega_0  \otimes \hat L ^\vee \otimes P)\cong ^{\rm proj.\ formula}\\H^i(A, \Omega _0 \otimes F^{es,*}_A(\hat L ^\vee \otimes P))\cong H^i(A, \Omega _0 \otimes F^{es,*}_A(\hat L ^\vee) \otimes P^{p^{es}})\end{multline*}
and (by Cohomology and Base Change) the required vanishing is equivalent to showing that 
$R\hat S (\Omega _0 \otimes F^{es,*}_A(\hat L ^\vee ) )$ is a sheaf (in degree $0$) for every $e \gg 0$.
(We have used the fact that $\hat A$ is $p$-divisible, so for any $Q\in \hat A$ there exists a $P\in \hat A$ with $Q=P^{p^{es}}$.)
This is equivalent to showing 
$$\hat \phi _{L,*}R  \hat S (\Omega _0 \otimes F^{es,*}_A(\hat L ^\vee ) )=R \hat S (\phi _{L}^*(\Omega _0 \otimes F^{es,*}_A(\hat L^\vee ) ))$$ is a sheaf for every $e \gg 0$ (cf. \eqref{m3.4}).
By Cohomology and Base Change, it suffices to show that $$H^i(\hat A, \phi _{L}^* \Omega _0 \otimes  \phi _{L}^*F^{es,*}_{A}(\hat L^\vee)  \otimes P )=H^i(\hat A, \phi _{L}^* (\Omega _0 \otimes  F^{es,*}_A (\hat L^\vee   ))\otimes P)=0$$ for $e \gg 0$, $i>0$ and $P\in \hat A$ (where $e$ is independent of $P$).
Since $F^{es}_A\circ \phi _L=\phi _L\circ  F^{es}_{ \hat A}$, we have $$\phi _{L}^*F^{es,*}_{A}(\hat L^\vee) =F^{es,*}_{ \hat A}\phi _{L}^*(\hat L^\vee ) =F^{es,*} _{\hat A}\left(\bigoplus _{h^0(L)}L\right)=\bigoplus _{h^0(L)}L^{p^{es}},$$ and so the last vanishing is immediate (for $e\gg 0$) from  Serre-Fujita vanishing \cite{Fujita}. 
\end{proof}

From now on we will adopt the following notation. 

\begin{notation}
\label{notation}
Let $\Omega_0$ be a coherent sheaf on $A$ and $s >0$ an integer, such that there is a  homomorphism $\alpha : \Omega_1:=F_*^s \Omega_0 \rightarrow \Omega_0$ (i.e., a Cartier module).  Unless otherwise specified, $\Omega_0$ is arbitrary. We also fix the following notation throughout the artice: $\Omega_e := F^{es}_* \Omega_0$,
$\Omega:= \varprojlim \Omega_e$, $\Lambda_e := R \hat S ( D_A( \Omega_e))$, $\Lambda := \varinjlim \Lambda_e$.
\end{notation}

\begin{cor}\label{c-0} With the  above notation, $\Lambda $ is a quasi-coherent sheaf and $\Omega =(-1_A)^* D_A RS(\Lambda )[-g]$. 
\end{cor}\begin{proof} By \eqref{l-ve} and  \eqref{GVT}.\end{proof}

\begin{proof}[Proof of \eqref{t-m}]
Choose $\Omega_0:= S^0 a_* \omega_X$. Then by \eqref{c-0},  $$0=\sH^i( \Lambda) = \sH^i( \hocolim R \hat S( D_A( F^e_* S^0 a_* \omega_X)) )=  \underbrace{\varinjlim \sH^i( R \hat S( D_A( F^e_* S^0 a_* \omega_X)) )}_{\textrm{\eqref{l-dir}}}.$$
\end{proof}

\subsection{Consequences of Theorem \ref{GVT}.} \label{ss-consequences_of_GVT}

First, we present a corollary that is not a consequence of \eqref{GVT}, but it is used frequently from here on. Then we list technical statements, most of which are used in Section \ref{sec:geometry}, except \eqref{lem:support} and \eqref{c-span} that are used already in \eqref{thm:closed_subvariety}. Note that the notations of \eqref{notation} are assumed from here.

\begin{cor} \label{c-1} 
For every closed point $y \in \hat A$, we have 
$$\Lambda \otimes k(y)\cong \varinjlim H^0(A,\Omega _e \otimes P_y^\vee)^\vee \cong \varinjlim H^0(A,\Omega _0 \otimes P_y^{-p^e})^\vee,$$ and for every closed point $y \in \hat A$ and integer $e \geq 0$, $$\sH^0(\Lambda_e) \otimes k(y)\cong H^0(A,\Omega _e \otimes P_y^\vee)^\vee \cong H^0(A,\Omega _0 \otimes P_y^{-p^e})^\vee.$$
\end{cor}
\begin{proof}Note first that $\Lambda_e$ is supported in cohomological degrees $[-g,\dots, 0]$ as explained in the proof of \eqref{GVT}. Hence, by cohomology and base change,\footnote{Recall that traditionally cohomology and base change is stated for cohomology of coherent sheaves, however it also applies for hypercohomologies of bounded complexes cf. \cite[7.7, 7.7.4, 7.7.12(ii)]{EGAIII} and the remark on \cite[3.6]{PP11}.} and for any $y\in \hat A$ we have $$\sH^0(\Lambda_e) \otimes k(y)=\sH^0(R\hat S D_A(\Omega _e ))\otimes k(y)=R^0\Gamma (D_A(\Omega _e )\otimes P_y)=H^0(A,\Omega _e\otimes P_y^\vee )^\vee .$$ 
Since $\Lambda =\varinjlim  \sH^0(R\hat S D_A(\Omega _e ))$, it follows by \eqref{l-dir} that $\Lambda \otimes k(y)\cong \varinjlim H^0(A,\Omega _e \otimes P_y^\vee)^\vee$. 
\end{proof}


\begin{cor}\label{c-4} Suppose that $ H^0(A,\Omega _0 \otimes P_y)=0$ for all $y\in \hat A$, then $\Lambda=0$ and $\Omega =0$. 
\end{cor}
\begin{proof} 
By \eqref{c-1}, $\Lambda_e=0$ for every $e$. Hence $\Lambda =0$ and then $\Omega = (-1_A)^* D_A RS(0) [-g]=0$ by \eqref{GVT}.
\end{proof}

Recall that a unipotent vector bundle is a given by finitely many successive extensions of line bundles $P\in \hat A$ or equivalently the Fourier Mukai transform of an Artinian module of finite rank on $\hat A$.
\begin{cor}\label{c-2} If
$\Lambda$ has a non-zero direct factor which is a direct limit of Artinian coherent sheaves, then  $RS(\Lambda)$ has a non-zero direct factor of the form $\varinjlim V_e$ where $V_e$ are unipotent vector bundles and the maps $V_e \to V_{e+1}$ are injective with cokernel being a unipotent vector bundle as well. In particular,  $\Supp \Omega_0 = \Supp D_A(RS(\Lambda))[-g]=A$  (recall that $ D_A(RS(\Lambda))[-g]$ is a sheaf by \eqref{c-0}).

\end{cor}
\begin{proof} 
By assumption, $\Lambda= B\oplus \varinjlim G_e $ where $G_e$ are Artinian sheaves, $\varinjlim G_e \neq 0$ and $B$ is some quasi-coherent sheaf. Then $ RS(\varinjlim G_e)$ is a direct factor of $RS(\Lambda)$. Thus, we may assume  that $\Lambda = \varinjlim G_e$.  Replacing $G_e$ by the image of $G_e \to \Lambda$, we may further assume that the maps $G_e \to G_{e+1}$ are injective. Now, since $G_e$ is Artinian, $H^i(\hat A \otimes_k k(x), G_e \otimes \sP_x )=0$ for every $i>0$ and $x \in A$. Therefore, by cohomology and base-change $V_e:=RS(G_e)$ is a vector bundle. Further there is an exact sequence 
\begin{equation*}
\xymatrix{
0 \ar[r] & V_e \ar[r] & V_{e+1} \ar[r] & W_e:=RS( \coker (G_e \to G_{e+1})) \ar[r] & 0
}
\end{equation*}
and both $V_e$ and $W_e$ are unipotent because Artinian coherent sheaves have a filtration by skyscraper sheaves of length one. Then by \eqref{l-dirRS}
\begin{equation}
\label{eq:RS_Lambda}
 RS ( \Lambda) \cong RS ( \varinjlim G_e) \cong \hocolim RS(G_e) = \varinjlim V_e,
\end{equation}
and furthermore since the homomorphism $V_e \to V_{e+1}$ are injective, $RS(\Lambda) \neq 0$. 

By \eqref{lem:support} and \eqref{c-0}, to prove the support statement it is enough to show that $\Supp D_A(RS (\Lambda))[-g] = A$.  So, by \eqref{eq:RS_Lambda} $D_A(RS (\Lambda)) [-g] = D_A( \varinjlim V_e) [-g] = \varprojlim D_A(V_e) [-g]$ (note that since the $V_e$ are unipotent vector bundles, $D_A(V_e)[-g]$ are sheaves), where the $D_A(V_e)[-g]$ fit into exact sequences of unipotent vector bundles
\begin{equation*}
\xymatrix{
0 \ar[r] & D_A(W_e)[-g] \ar[r] & D_A(V_{e+1})[-g] \ar[r] & D_A(V_e)[-g] \ar[r] & 0.
}
\end{equation*}
Therefore $D_A(RS (\Lambda))[-g]$ is an inverse limit of unipotent vector bundles with surjective maps between them. Thus, $\Supp D_A(RS (\Lambda))[-g] = A$.
\end{proof}

\begin{lem}
\label{lem:support}
If $\alpha : F_*^s \Omega_0 \to \Omega_0$ is surjective, then $\Supp \Omega= \Supp \Omega_0$. 
\end{lem}

\begin{proof}
Let $P \in A$. There are two cases:
\begin{itemize}
\item If $P \notin \Supp \Omega_0$, then for every open set $U \subseteq A \setminus \Supp \Omega_0$, $\Omega_0(U) = 0$. Therefore also $\Omega_e(U) = 0$. Hence $(\varprojlim \Omega_e)(U) = \varprojlim (\Omega_e(U)) = 0$. In particular, $\Omega_P=0$, and therefore $P \notin \Supp \Omega$.
\item If $P \in \Supp \Omega_0$, then choose an affine open set $U \ni P$ and an element $s_0 \in \Omega_0(U)$, such that its image in $(\Omega_0)_P$ is not zero. Since $U$ is affine, there is a chain of elements $s_e \in \Omega_e(U)$ such that $s_e$ maps onto $s_{e-1}$ for each $e >0$. Therefore $(s_e|e \geq 0)\in \varprojlim \Omega_e(U) = (\varprojlim \Omega_e)(U)$ defines an element the restriction of which to any $V \subseteq U$ is not zero, because $(s_e|e \geq 0)|_V = (s_e|_V | e \geq 0)$ and $s_0|_V \neq 0$ by the choice of $s_0$. Therefore, this defines a non-zero element of $\Omega_P$, which shows that $P \in \Supp \Omega$. 
\end{itemize}
\end{proof}

\begin{cor}\label{c-span} 
Assume that $F_*^s \Omega_0 \to \Omega_0$ is surjective.  Let $\hat B\subset \hat A$ be an abelian subvariety such that  $$V^0(\Omega_0)=\{P\in \hat A | h^0(\Omega_0 \otimes P)\ne 0\}$$ is contained in finitely many translates of $\hat B$. Then $T_x^* \Omega \cong \Omega$ for every $x \in \widehat{\hat A /\hat B}$. In particular $ \Supp \Omega $ (which is a closed subvariety by \eqref{lem:support}) is fibered by the projection $A\to B$ (i.e., $\Supp \Omega$ is a union of fibers of $A \to B$).
\end{cor}
\begin{proof} By \eqref{c-1}, the sheaf $\sH ^0(\Lambda_0)$ is supported on $V^0(\Omega_0)$. Let $\hat K=\hat {A}/\hat B$, then as $V^0(\Omega_0)$ is contained in finitely many fibers of  $\pi :\hat {A}\to \hat K$, it follows that $\sH ^0(\Lambda_0)\otimes \pi^* P\cong \sH ^0(\Lambda_0)$ for all $P\in {\rm Pic}^0(\hat K)=K\subset   {\rm Pic}^0(\hat A )=A$.
Since $\hat F^{es}:\hat A \to \hat A$ is an isogeny, for any $ P\in {\rm Pic} ^0(\hat A)$ and any $e>0$ we may pick $ Q\in {\rm Pic }^0(\hat A)$ such that $\hat F^{es,*} Q\cong P$. If moreover $ P\in \pi ^* {\rm Pic} ^0(\hat K)$, then we may assume that $Q\in \pi ^* {\rm Pic} ^0(\hat K)$.
By \eqref{m3.4}, it follows that 
$$\sH^0(\Lambda _e)\otimes P\cong (\hat F ^{es,*} \sH^0(\Lambda _0))\otimes P\cong \hat F ^{es,*}(\sH^0(\Lambda _0) \otimes Q)\cong \hat F ^{es,*} \sH^0(\Lambda _0) \cong \sH^0(\Lambda _e).$$  But then 
\begin{equation}
\label{eq:Lambda_invariant}
\Lambda \otimes P = \sH^0(\Lambda ) \otimes P =\varinjlim \sH^0(\Lambda _e) \otimes P \cong \Lambda  
\end{equation}
 and so
\begin{multline*}
T_x^* \Omega \cong T_x^*((-1_A)^* (D_A RS (\Lambda )) [-g]) \cong
\\ \cong ((-1_A)^*  T_{-x}^* D_A RS (\Lambda )) [-g] \cong \underbrace{((-1_A)^*   D_A T_{-x}^* RS (\Lambda )) [-g]}_{\textrm{ \eqref{lem:D_A_T_x}}}
\\ \cong \underbrace{((-1_A)^*   D_A  RS (\Lambda \otimes P_{x} )) [-g]}_{\textrm{\eqref{M3.1}}} \cong \underbrace{((-1_A)^*   D_A  RS (\Lambda  )) [-g]}_{\textrm{\eqref{eq:Lambda_invariant}}} \cong \Omega .
\end{multline*}
Therefore $\Supp \Omega$  is invariant under $T_x$ for every $x \in K$, which concludes our proof.
\end{proof} 

\begin{rem}
The image of $\Supp \Omega$ in $B$ can have positive dimension, even if one takes $\hat B$ to be the smallest abelian subvariety as above. For example if one takes the embedding $a : C \to A$ of a curve  of genus at least two into its Jacobian, and $\Omega_0:= S^0 a_* \omega_C$. Then $\Omega_0 = a_* \omega_C$ and $\Supp \Lambda_0 = \hat A$. Thus, $\hat A=  \hat B$, and hence $A=B$. So, the image of $\Supp \Omega = \Supp \Omega_0$ in $B$ is isomorphic to $C$. 
\end{rem}
\begin{thm}
\label{thm:closed_subvariety}
If $a : X \hookrightarrow A$ is a closed, smooth subvariety of general type of an abelian variety, then the smallest abelian subvariety $\hat B \subseteq \hat A$ such that the union of finitely many translates of $\hat B$ contains $V^0(A, S^0 a_* \omega_X)$ is equal to $\hat A$.
\end{thm}

\begin{proof}
Assume that $\hat B \subsetneq \hat A$. Set $\Omega_0= S^0 a_* \omega_X$ ($=a_* \omega_X$, because $X$ is smooth). By \eqref{c-span} and \eqref{lem:support}, there is a fibration $h : X \to Y$ such that every fiber is a positive dimensional abelian variety. Hence $\omega_X|_G \cong \sO_G$ for the general fiber $G$ of $h$. However, this contradicts the fact that $\omega_X$ is big. 
\end{proof}

\begin{rem}
It is easy to generalize \eqref{thm:closed_subvariety} to the case when $X$ is mildly singular in an adequate sense. We leave this to the interested reader since we will not need this in what follows.
\end{rem}

\subsection{Frobenius stable cohomology support loci (proof of \eqref{c-gv-intro} and \eqref{p-finite1-intro})}\label{ss-cohomo_support_loci}

We do not have an optimal definition of Frobenius stable cohomology support locus (see below for the different variants). So, we present separate statements for a few different possible candidates. Recall that \eqref{notation} is assumed for this section and the rest of the article.

\begin{cor}\label{c-gv} There exists a proper closed subset $Z\subset \hat A$ such that if $i>0$ and $p^e y\not \in Z$ for infinitely many $e\geq 0$, then $\varinjlim H^i(A_{k(y)},\Omega _e \otimes P_y^\vee)^\vee=0$ or equivalently $\varprojlim H^i(A_{k(y)},\Omega _e\otimes P_y^\vee)=0$. Let $W^i=\{ y\in \hat A| \varprojlim H^i (A_{k(y)}, \Omega _e\otimes P_y^\vee)\ne 0\}$, then $W^i\subset Z'=\overline{\bigcup \left( [{p^e_{\hat A}}]^{-1}(Z)\right)_{\rm red}}$. 
\end{cor}
\begin{proof} Let $Z$ be the proper closed subset where  $\sH^i(\Lambda_0)$ is not locally free  for any $0\leq i\leq g$. Note that $\Lambda_e = V^{e,*} \Lambda_0$, where $V$ is the Verschiebung which takes $P_y$ to $P_y^{p^e}=P_{p^ey}$ (cf. \eqref{m3.4} and \eqref{lem:dual_of_Frobenius}). In particular, if $y$ is as above, that is, $p^e y \not\in Z$ for infinitely many $e \geq 0$, then $\sH^i(\Lambda_e)$ is locally free at $y$ for infinitely many $e \geq 0$ and every $i$. In particular, by cohomology and base change (similarly to the proof of \eqref{c-1}), for infinitely many $e \geq 0$,
$$\sH^{-i}(\Lambda_e)\otimes k(y)\cong H^i(A_{k(y)},\Omega _e \otimes P_y^\vee)^\vee.$$ 
Since $\varinjlim \sH^{-i}(\Lambda_e)=0$ by \eqref{GVT}, it follows that $\varinjlim H^i(A_{k(y)},\Omega _e \otimes P_y^\vee)^\vee=0$.
Note that $\varprojlim H^i(A_{k(y)},\Omega _e\otimes P_y^\vee)=D_{k(y)}(\varinjlim H^i(A_{k(y)},\Omega _e \otimes P_y^\vee)^\vee)$.

\end{proof}

\begin{proposition}\label{p-finite} Let $X$ be a smooth, projective variety over $k$, $a : X \to A$ the Albanese morphism of $X$ and define $ V_S^0 :=\{P\in \hat A | S^0(X,\omega _X\otimes a^*P)\ne 0\}$. Then:
\begin{enumerate}
 \item $V_S^0 \subset \overline{V_S^0 }$ is the complement of countably many locally closed subsets. 
\item Whenever $P\in V_S^0 $, we also have $P^p\in V_S^0 $. If moreover $A$ has no supersingular factors, then each maximal dimensional irreducible component of the closure of $V_S^0 $ is a torsion translate of an abelian subvariety of $\hat A$.
\end{enumerate}
\end{proposition}
\begin{proof}
We first prove point $(2)$.
Define 
\begin{equation*}
S_{e,1} : = \im ( H^0(X, F^e_*(\omega_X \otimes a^* P^{p^e}) ) \to  H^0(X, F_*(\omega_X \otimes a^* P^{p} )) ) .
\end{equation*}
Then, 
\begin{equation*}
H^0(X, F_*\omega_X \otimes a^* P^p) = S_{1,1} \supseteq S_{2,1} \supseteq S_{3,1} \supseteq \dots. 
\end{equation*}
Suppose that $P\in V^0_S$. Since $\mathrm{Tr} (S_{e,1}) = S^0(X, \omega _X\otimes a^* P)\ne 0$ for every $e \gg 0$, $S_{e,1} \neq0$ for every integer $e>0$. 
Since pushing forward via $F$ induces isomorphisms $H^0(X, F^{e-1}_*(\omega_X \otimes a^* (P^p)^{p^{e-1}}) )\cong H^0(X, F^e_*(\omega_X \otimes a^* P^{p^e} )) $, it follows that 
\begin{equation*}
S_{e,1} \cong  \im ( H^0(X, F^{e-1}_*(\omega_X \otimes (a^* P^p)^{p^{e-1}}) ) \to  H^0(X, \omega_X \otimes a^* P^{p} ) ) .
\end{equation*}
 Thus $S^0(X, \omega_X \otimes a^* P^p) \neq 0$.
By \eqref{t-pr}, if $A$ has no supersingular factors, one sees that maximal dimensional irreducible component of  $\overline{V_S^0 } $ is a finite union of torsion translates of abelian subvarieties of $\hat A$.

To prove point $(1)$, let $Z$ be an irreducible component of $\overline{V_S^0 } $. Note that for any $e>0$, the set of $P\in Z$ such that the image of $H^0(X, F^e_*(\omega_X \otimes a^* P^{p^e}) ) \to  H^0(X, \omega_X \otimes a^* P)$ is non-zero is a constructible subset. Thus $V_S^0 \subset \overline{V_S^0 }$ is the complement of countably many locally closed subsets.
\end{proof}
\begin{proposition}\label{p-finite1} With assumptions as in \eqref{p-finite} (including that $A$ has no supersingular factors),  each maximal dimensional irreducible component of the closure of the set of points such that $\varinjlim H^0(X, F^{e}_*\omega _X\otimes a^* P_y )^\vee \ne 0$ is a  torsion translate of an abelian subvariety of $\hat A$.
\end{proposition}
\begin{proof} Since $$H^0( X,(F^{e}_*\omega _X)\otimes a^*P ) \cong H^0( X,F^{e}_*(\omega _X\otimes a^*P^{p^e}) ) \cong H^0( X,F^{e-1}_*(\omega _X\otimes a^*(P^p)^{p^{e-1}}) ),$$ it is easy to see that if $\varinjlim H^0( X,(F^{e}_*\omega _X) \otimes a^*P )^\vee \ne 0$ then $\varinjlim H^0(X, (F^{e}_*\omega _X)\otimes a^*P^p )^\vee \ne 0$. The proof now follows along the lines of the previous proposition.
\end{proof}

\begin{corollary}\label{c-finite} With assumptions as in \eqref{p-finite} (including that $A$ has no supersingular factors), assume also that $\kappa _S (X)\leq 0$ (See Section \ref{ss-K_S} for the definition) and that $k$ is uncountable. Then $V_S^0 $ contains at most one point. 
\end{corollary}
\begin{proof} Let $T+P\subset \overline{V_S^0} $ where $\dim T>0$ is maximal and $P\in \hat A$ is torsion and $T\subset \hat A$ is an abelian subvariety. Pick $m\geq 2$ such that $P^m\cong \mathcal O_A$.
Then for any very general $Q\in T$ we have a map $$S^0(\omega _X\otimes Q\otimes P)^{m-1}\otimes S^0(\omega _X\otimes Q^{-m+1}\otimes P)\to S^0(\omega _X^m ).$$
It follows immediately that $\dim S^0(\omega _X^m)\geq 2$. This is impossible and hence $\dim V_S^0 =0$ and so $V_S^0 $ is a union of finitely many torsion points. Therefore, every component of $\overline{V_S^0} $ is zero dimensional and so $\overline{V_S^0} ={V_S^0} $.

By a similar argument to the one above it then follows that $V_S^0 $ contains at most one point. Suppose by way of contradiction that there are two elements $P\ne Q$ in $V_S^0 $. By what we have seen above, $P,Q\in \hat A$ are torsion points and so there is an integer $m>0$ such that $P^m=Q^m=\mathcal O _X$. Let $G_P\in |K_X+P|$, $G_Q\in |K_X+Q|$ be corresponding divisors, then $mG_P,mG_Q\in |mK_X|$ are distinct divisors corresponding to elements of $S^0(\omega ^m_X)$ so that $\dim S^0(\omega _X^m)\geq 2$. This is the required contradiction.\end{proof}
\begin{proposition}\label{p-lambda} Let $A$ be an abelian variety that has  no supersingular factors, then each maximal dimensional irreducible component of the set $Z$ of points $P\in \hat A$ such that  the image of $\mathcal H^0(\Lambda _0)_P \to \Lambda_P$ is non-zero,  is a  torsion translate of an abelian subvariety of $\hat A$ and  $\Lambda_P\ne 0$ if and only if $P^e\in Z$ for some $e>0$.
\end{proposition}
\begin{proof} Let $V:\hat A\to \hat A$ be Verschiebung so that $V([P])= [P^p]$ for every $[P] \in \hat A$. Let $V_P:\Spec \mathcal O _{\hat A,P}\to \Spec \mathcal O _{\hat A,P^p}$ be the induced morphism. We have $$V_P^*(\Lambda_{P^p})=V_P^*(\varinjlim (V^{e,*} \mathcal H^0(\Lambda_e))_{P^{p}} )=\varinjlim (V^{e+1,*}\mathcal H^0(\Lambda_e))_P=\Lambda_P.$$ It follows that if $\Lambda_P \ne 0$ then also $\Lambda_{P^{p^e}}\ne 0$.
Let $K_e$ denote the kernel of  $\mathcal H^0(\Lambda _0)\to \Lambda_e$, then $K_i\subset K_{i+1}\subset \ldots$ so that $K_i=K$ for all $i\gg 0$. Therefore, the image of $\mathcal H^0(\Lambda _0)\to \Lambda$ 
is a coherent sheaf (isomorphic to $\mathcal H^0(\Lambda _0)/K$). It follows that 
if $Z$ is the set of  $P\in \hat A$ such that the image of $\mathcal H^0(\Lambda _0)_P \to \Lambda_P =\varinjlim ( V^{e,*}\mathcal H^0(\Lambda _0)_P)$ is non-zero, then $Z$ is a closed subset of $\hat A$.  
Since $V$ is faithfully flat, so is $V_P$.  Thus if the map $$\mathcal H^0(\Lambda _0)_P\to (V^*\mathcal H^0(\Lambda _0))_P=V_P^*(\mathcal H^0(\Lambda _0)_{P^p})\to V_P^*(\Lambda_{P^p})=\Lambda_P$$ is non-zero then the map $\mathcal H^0(\Lambda _0)_{P^p}\to \Lambda_{P^p}$ is non-zero as well. 
It follows that that if $P\in Z$, then $P^p\in Z$.
If $A$ has  no supersingular factors, then the claim follows from \eqref{t-pr}. Finally, since $\Lambda_P \ne 0$ if and only if  the image of $\mathcal H^0(\Lambda _e)_P \to \Lambda_P$ is non-zero for some $e \geq 0$, it follows that $\Lambda \otimes \mathcal O _{\hat A,P}\ne 0$ if and only if $P^e\in Z$ for some $e\geq 0$.
\end{proof}

\subsection{Examples} \label{ss-examples}

We begin by showing that for an ordinary abelian variety $A$ and for an integer $0 \leq i \leq \dim A$,
\begin{equation*}
 \overline{V^i(\Omega)} = \hat A
\end{equation*}
where $\Omega= \varprojlim \Omega_e$, $\Omega_e=F^e_* \omega_A$ and 
\begin{equation*}
V^i(\Omega):=\{ [P] \in \hat A  |  H^i(A, \Omega \otimes P) \neq 0 \}.
\end{equation*}

\begin{proposition}
For an ordinary abelian variety $A$, we have
\begin{equation*}
\{[P] \in \hat A | \exists e >0 : P^{p^e} \cong \sO_A  \} = V^i(\Omega) 
\end{equation*}
Note that this is a countably infinite set by \cite[Application 2, page 62]{Mumford}.
\end{proposition}

\begin{proof}
For every $[P] \in \hat A$,
\begin{equation*}
H^i(A,  \Omega \otimes P) 
= \underbrace{ \varprojlim H^i(A, \Omega_e \otimes P)}_{\textrm{\cite[13.3.1]{EGAIII}}} = \varprojlim H^i(A, \omega_A \otimes P^{p^e}).
\end{equation*}
In particular, if $P$ is $p$-power torsion then 
\begin{equation*}
\varprojlim H^i(A, \omega_A \otimes P^{p^e}) 
\cong \underbrace{\varprojlim H^i(A, \omega_A )}_{\parbox{100pt}{\tiny discarding finitely many terms and using that $P^{p^e}\cong \sO_A$ for every $e \gg 0$}}
\cong \underbrace{(\varinjlim H^i(A, \sO_A ))^\vee}_{\textrm{Serre duality}}
\end{equation*}
where the map between the countably many copies of $H^i(A, \sO_A )$ in the last direct limit is the natural homomorphism induced by the Frobenius. In particular, this homomorphism is bijective, because $A$ is ordinary. Hence, 
\begin{equation*}
\varprojlim H^i(A, \omega_A \otimes P^{p^e}) \cong H^i(A, \sO_A)^\vee \neq 0  . 
\end{equation*}
If $P$ is not $p^e$ torsion for any $e\geq 0$, then  $H^i(A, \omega_A \otimes P^{p^e}) =0$ for all $e\geq 0$ and so $H^i(A,  \Omega \otimes P) =0$.
This concludes our proof.
 
\end{proof}

\begin{proposition} \cite[(5.30)]{MvdG}
In the situation of the above proposition, $V^i(\Omega)$ is dense in $\hat A$.
\end{proposition}

Hence the following seems to be the most natural question.

\begin{q}
Is $V^i(\Omega)$ a countable union of Zariski closed sets with codimension at least $i$?
\end{q}

We now compute examples of $\Lambda$. 

\begin{ex}
Let $E$ be an elliptic curve and $\Omega_0 := \omega_E$. There are two cases:

If $E$ is supersingular (which is equivalent for elliptic curves to being not ordinary), then note that $ R\hat S(\sO _E[g] )=k_{0_{\hat E}}=\Lambda _0$ and $ R\hat S(F_*\sO _E[g])=\Lambda _1=\hat F^*\Lambda _0$ is Artinian of length $p$. By \eqref{c-1},
$\Lambda _1 \otimes k_{0_{\hat E}} \cong k_{0_{\hat E}}$. So, $\Lambda_1$ is an Artinian $\hat \sO _{\hat E, 0} \cong k[[x]]$ module which has dimension one when tensored with the residue field. Therefore, $\Lambda_1 \cong k[[x]]/(x^p)$ as a $k[[x]]$ module. Similarly $\Lambda_e \cong k[[x]]/(x^{p^e})$.

Now, let us determine the map $\Lambda_0 \to \Lambda_1$. It is a $k[[x]]$-module homomorphism $k \to k[[x]]/(x^p)$. Up to a multiplication by a unit (which can be disregarded for our purposes) there are two such maps: the zero map, and the multiplication by $x^{p-1}$. Since $\Lambda_e \to \Lambda_{e+1}$ is obtained by applying $\hat F ^{e,*}$ to $\Lambda_0 \to \Lambda_1$, if the latter was zero, then all the maps $\Lambda_e \to \Lambda_{e+1}$ would be zero, and consequently also $\Lambda$ would be zero. This is impossible by \eqref{c-0}, because $\Omega$ is not zero in our situation. Hence $\Lambda_0 \to \Lambda_1$ has to be the multiplication by $x^{p-1}$ map. In particular it is injective. Since $\hat F ^e$ is faithfully flat, $\Lambda_e \to \Lambda_{e+1}$ then has to be also injective. Therefore, $\Lambda_e \to \Lambda_{e+1}$ can be identified with the map $k[[x]]/(x^{p^e}) \to k[[x]]/(x^{p^{e+1}})$ given by multiplication by $x^{p^e(p-1)}$. The quasi-coherent sheaf $\Lambda$ is then the direct limit of 
the modules $k[[x]]/(x^{p^e})$ viewed as a direct system via  multiplications by $x^{p^e(p-1)}$. Note that this is a torsion $k[[x]]$-module, and $\Lambda \otimes_{k[[x]]} k_{0_{\hat E}}=0$ (though $\Lambda \neq 0$). Further, $\Supp \Lambda = \{ 0_{\hat E} \}$.

If  $E$ is an ordinary elliptic curve, then the induced map $H^1(\sO _E)\to H^1(F_*\sO _E)$ is an isomorphism and  $\hat F = V$ is the \'etale map that sends each $Q \in \hat E$ to $Q^p$. It follows that $\Lambda_e= \bigoplus_{Q^{p^e} = \sO_E} k(Q)$, and the maps $\Lambda_e \to \Lambda_{e+1}$ are the natural embeddings. Therefore $\Lambda =\bigoplus _{y\in E[p^\infty]}k(y)$ where $E[p^\infty]$ denotes the set of all $p^\infty$ torsion points in $\hat E$. In particular, $\Supp \Lambda = E[p^\infty]$, which is a countable infinite set.

\end{ex}

\section{Geometric consequences}
\label{sec:geometry}

\subsection{Frobenius stable Kodaira dimension}
\label{ss-K_S}

In this section $X$ is always a smooth, projective variety over $k$. In characteristic $p>0$ the space $S^0(X,\mathcal O _X(mK_X))$ is better behaved than $H^0(X,\mathcal O _X(mK_X))$. So we define $$S(K_X)=\bigoplus _{m\geq 0}S^0(X,\mathcal O _X(mK_X))\subset  R(K_X)=\bigoplus _{m\geq 0}H^0(X,\mathcal O _X(mK_X))$$
and
$$\kappa _S(X)={\rm max} \{ k|\dim S^0(X,\mathcal O _X(mK_X))=O(m^k) \ {\rm for}\ m\ {\rm sufficiently \ divisible}\}.$$
It is easy to see that $S(K_X)$ is a birational invariant.
\begin{lem}\label{l-rm} $S(K_X)\subset R(K_X)$ is an ideal.\end{lem}

\begin{proof}If $f\in S^0(X,\mathcal O _X(mK_X))$ and $g\in H^0(X,\mathcal O _X(lK_X))$ then $fg\in S^0(X,\mathcal O _X((m+l)K_X))$ in fact there are $f_e\in H^0(X,\mathcal O _X((1+(m-1)p^e)K_X))$ such that $\Phi ^eF^e_*(f_e)=f$ and by the projection formula we have $f_eg^{p^e}\in  H^0(X,\mathcal O _X((1+(m+l-1)p^e)K_X))$ such that $\Phi ^eF^e_*(f_eg^{p^e})=fg$.
\end{proof}

\begin{rmk} Note that  for curves we have $S(K_{\bP^1}) =0$, $S(K_E) =0$ (resp.  $S(K_E) =k[x]$) if $E$ is  a supersingular (resp. ordinary) elliptic curve and if $X$ is a curve of genus at least two, then $S(K_X)_n= R(K_X)_n$ for $n \gg0$, since $K_X$ is ample \cite[Corollary 2.23]{Pat12}.  In higher dimensions, assuming the finite generation of $R(K_X)$, it then follows that the ideal $S(K_X)$ is a finitely generated $R(K_X)$ module. In particular this holds in dimension $2$. \end{rmk}

\begin{lem} \label{lem:kappa_vs_kappa_S} If $\kappa _S(X)\geq 0$, then $\kappa (X)=\kappa _S(X)$.\end{lem}
\begin{proof} Since $S^0(X,\mathcal O _X(mK_X))\subset H^0(X,\mathcal O _X(mK_X))$ the inequality  $\kappa (X)\geq \kappa _S(X)$ is clear. The reverse inequality is immediate from the fact that $S(K_X)$ is a torsion free module over the integral domain $R(K_X)$ and hence there are many (module) embeddings $R(K_X) \hookrightarrow S(K_X)$.
\end{proof}
\begin{rmk} Note however that if $Y$ is a supersingular elliptic curve, then $\kappa _S(Y)=-\infty$ but $\kappa (Y)=0$. If $Z$ is a variety of general type, then for $X = Y \times Z$, $k(X)= \dim X -1$. However $S^0(X, \omega_X^m) = S^0(Y, \omega_Y^m) \otimes S^0(Z, \omega_Z^m)$ (c.f., \cite[Lemma 2.31]{Pat12}), and so $S^0(X, \omega_X^m)=0$ for every $m>0$. Thus $\kappa_S(X)=-\infty$ and $\kappa(X)=\dim X -1$.\end{rmk}
\begin{lem}\label{l-lim} We have $$\lim \dim S^0(X,\mathcal O _X(mK_X))/m^n=\lim \dim H^0(X,\mathcal O _X(mK_X))/m^n.$$ In particular $\kappa (X)=\dim X$ iff $\kappa _S(X)=\dim X$.\end{lem}

\begin{proof} If $\kappa(X)= \dim X$, then there is a very ample line bundle $\sA$ such that $S^0(X,\sA)\ne 0$ and an integer $N>0$, such that $\sA \hookrightarrow \omega_X^N$. Therefore, multiplying by a nonzero section in $S^0(X,\sA)$, we have $$H^0(X,\omega_X^m )\hookrightarrow S^0(X, \omega_X^m \otimes \sA) \hookrightarrow S^0(X, \omega_X^{m+N}).$$ 
We thus have inequalities $$h^0 (X,\omega_X^m )\leq \dim S^0(X, \omega_X^{m+N})\leq h^0(X, \omega_X^{m+N})$$ and the claim follows immediately.\end{proof}
\begin{lem} The limit  $$\lim _{1||m\to +\infty}\frac{\dim S^0(X,\mathcal O _X(mK_X))}{m^{\kappa _S(X)}}$$ exists. (Here we have assumed that $1||m$ i.e. that $m>0$ is sufficiently divisible.) \end{lem}
\begin{proof} If $\dim S^0(X,\mathcal O _X(mK_X))=0$ for all $m\geq 0$ then the result trivially holds. Thus we may assume that
$\dim S^0(X,\mathcal O _X(mK_X))>0$ for some $m\geq 0$. Arguing as in the proof of \eqref{l-lim}, we have 
$$h^0 (X,\omega_X^m )\leq \dim S^0(X, \omega_X^{m+N})\leq h^0(X, \omega_X^{m+N})$$ for all $m,N>0$ sufficiently divisible. The claim now follows since by a result of Kaveh and  Khovanskii \cite{KK12}, we have  that $$\lim _{1||m\to +\infty}\frac{h^0(X,\mathcal O _X(mK_X))}{m^{\kappa _S(X)}}$$ exists.
\end{proof}


\subsection{Proof of point $(1)$ of Theorem \ref{t-one}} \label{ss-thm_1}

\begin{defn}
Let $X$ be a smooth projective variety over $k$ such that $\kappa(X)=0$. The \emph{Calabi-Yau index} is then defined to be the greatest common divisor of the integers $m>0$, for which $h^0(mK_X) \neq 0$. By the following remark, it is also the smallest integer $m>0$, for which $h^0(mK_X) \neq 0$. 
\end{defn}

\begin{rem}
From the Chinese remainder theorem it follows that if $r$ is the Calabi-Yau index  of $X$, then there is an integer $l > 0$, such that $h^0(rl K_X) \neq 0$ and $h^0(r(l+1)K_X) \neq 0$. Let $D \in |rl K_X|$ and $D' \in |r(l+1)K_X|$ be the unique elements for some $l \gg 0$. Then $(l+1)D = l D'$. In particular $\frac{l+1}{l} D = D'$ and hence $D' - D = \frac{1}{l}D$ is an effective $\bZ$-divisor, an element of $|r K_X|$. 
\end{rem}

\begin{rem}
Note also that if $\kappa_S(X) = 0$, then $\kappa(X)=0$ by Lemma \ref{lem:kappa_vs_kappa_S}, and hence the Calabi-Yau index is defined.
\end{rem}

\begin{lem}
\label{lem:Calabi_Yau_index_divides_p_minus_1}
If $X$ is a smooth projective variety with $\kappa_S(X) = 0$, then the Calabi-Yau index of $X$ divides $p-1$.
\end{lem}

\begin{proof} Let and $m>0$ be an integer such that $S^0(X, \omega_X^m) \neq 0$ and
 $r$ be the Calabi-Yau index, then $r | m$. Furthermore, since for every $e>0$ there is a non-zero element of $H^0(X, \sO_X((p^e m + (1-p^e)) K_X))$, where
\begin{equation*}
0 \equiv p^e m + (1-p^e) \equiv (1-p^e) \mod r .
\end{equation*}
For $e=1$ we obtain that $r|p-1$.
\end{proof}

\begin{lem}
\label{lem:S_0_if_Kodaira_dimension_zero}
If $X$ is a smooth projective variety such that $\kappa_S(X) = 0$ with Calabi-Yau index $r$ and $a : X \to A$ is morphism to a projective scheme $A$ over $k$, then 
\begin{enumerate}
\item $S^0(X, \omega_X^r) \neq 0$,
\item if $G \in |rK_X|$ is the unique element, then $S^0(X, \sigma(X,\Delta) \otimes \omega_X^r) \neq0$, where $\Delta = \frac{r-1}{r} G$,
\item the natural inclusions $S^0(X, \sigma(X,\Delta) \otimes \omega_X^r) \subseteq H^0(A, \Omega_0) \subseteq H^0(X, \omega_X^r)$ are equalities, where $\Omega_0= S^0 a_* ( \sigma(X,\Delta) \otimes \omega_X^r)$ and $\Delta$ is as above and the natural inclusion is obtained from \eqref{lem:obvious_inclusion}.
\item the natural maps  $H^0(A, F^{e+1}_* \Omega_0) \to H^0(A, F^{e}_* \Omega_0)$ and $H^0(X, F^{e+1}_* \omega_X^r) \to  H^0(X, F^{e}_* \omega_X^r)$ are compatible with the above inclusions and hence both are isomorphisms.
\end{enumerate}

\end{lem}

\begin{proof}
For point $(1)$, let $0\ne f \in H^0(X, \sO _X(rK_X))$ corresponding to a divisor $G$. Then, there is an integer $l>0$, such that $S^0(X,\omega _X^{rl})\ne 0$ \eqref{lem:kappa_vs_kappa_S}. Thus, for all integers $e >0$, there is an element of $H^0(X, \mathcal O _X((p^erl +(1-p^e) ) K_X) )$ mapping to $0 \neq f^l \in  H^0(X, \mathcal O _X(rl K_X))$. Since $\kappa (X)=0$, that element can only be $\alpha f^{p^el+\frac{1-p^e}{r}}$ for some $\alpha \in k^*$. Now, let us look at $\Phi^e \left(F_*^e \left(\alpha f^{p^e+\frac{1-p^e}{r}} \right) \right)$. If it were zero, then the following element would also be zero
\begin{equation*}
\Phi^e \left(F_*^e \left(\alpha f^{p^e+\frac{1-p^e}{r}} \right) \right) f^{l-1} =  \Phi^e \left(F_*^e \left(\alpha f^{lp^e+\frac{1-p^e}{r}} \right) \right) = f^l .
\end{equation*}
But we know that $f^l$ is not-zero, so also $\Phi^e \left(F_*^e \left(\alpha f^{p^e+\frac{1-p^e}{r}} \right) \right)$ is not zero and hence equals $\beta f$ for some $\beta \in k^*$. Therefore, $f \in \im \Phi^e$ for every $e>0$ and hence $f \in S^0(X,\sO_X(r K_X))$.

Point $(2)$ follows from the fact that that the image of 
\begin{equation*}
\alpha f \in H^0(X,\sO_X(rK_X)) = H^0( X, \sO_X( rp^e K_X + (1-p^e)(K_X + \Delta))) 
\end{equation*}
 in $H^0(X,\sO_X(rK_X))$ is computed by
\begin{equation*}
\alpha f \mapsto \alpha  f \cdot f^{\frac{r-1}{r} (p^e -1)} = \alpha f^{p^e + \frac{1-p^e}{r}} \mapsto   \beta f .
\end{equation*}

Point $(3)$ immediately follows from the fact that $S^0(X, \omega_X^r) \neq 0$, and that $\dim_k H^0(X, \omega_X^r)=1$. 

Point $(4)$ follows from $H^0(A, \_)$ applied to the  commutative diagram
\begin{equation*}
\xymatrix{
F^{e+1}_* \Omega_0 \ar[r] \ar@{^(->}[d] & F^e_* \Omega_0 \ar@{^(->}[d] \\
F^{e+1}_* a_* \omega_X^r \ar[r] \ar@{=}[d] & F^e_* a_* \omega_X^r  \ar@{=}[d] \\
a_* F^{e+1}_* \omega_X^r \ar[r]  & a_* F^e_* \omega_X^r
},
\end{equation*}
where the bottom two horizontal arrows are the arrows from \eqref{lem:S_f_*_mK_X}. Further the bottom one is an isomorphism because the stable image of these maps is exactly $S^0(X, \sigma(X,\Delta) \otimes \omega_X^r)$, which is proven to be non-zero in point $(2)$.
\end{proof}

\begin{thm} Let $a:X\to A$ be a generically smooth morphism from a smooth projective variety to (but not necessarily onto) an abelian variety with general fiber $G$.
If $S^0(G,\omega _G ) \neq 0$, then $H^0(X,\omega _X\otimes a^* P) > 0$ for some $P\in \hat A$.
\end{thm}
\begin{proof} Since $S^0(G,\omega _G) \neq 0$, it follows by \eqref{thm:openness_S_0} that $S^0a_*\omega _X\ne 0$. If $$H^0(X,\omega _X\otimes a^* P)= 0,\qquad \forall\ P\in \hat A,$$ then 
$\Omega =0$ by \eqref{c-4}, but since $\Omega \to S^0a_*\omega _X\ne 0$ is surjective, this is impossible.
\end{proof}

\begin{rem}
Note that the above theorem does not hold if one replaces $H^0$ by $S^0$. That is, there are examples where $S^0(X,\omega _X\otimes a^* P)= 0$ for every $P \in \hat A$. An easy example is when $X=A$ is a non-ordinary abelian variety, and $a= \id_A$. Then $S^0(X, \omega_X \otimes a^* P ) = S^0(A, \omega_A \otimes P)$, which is zero for $P \neq \sO_A$ because then $H^0(A, \omega_A \otimes P) =0$ and it is zero for $P = \sO_A$, because $A$ is not ordinary.
\end{rem}
                                      
\begin{thm} Let $X$ be a smooth projective variety defined over $k$ and $a:X\to A$ the Albanese morphism. Suppose that $S^0a_*\omega _X\ne 0$ and ${\rm Pic }^0(X)$ has no supersingular factors, then if $\kappa _\nu (X)=0$, then $\kappa (X)=0$. \end{thm} \begin{proof} 
By \cite{CHMS12}, we know that $K_X\equiv G$ for the effective $\mathbb R$-divisor $G=N_\sigma (K_X)$. Assume now that $P' \in V^0(\omega_X)$, that is, $h^0(K_X + a^* P') \neq 0$. Then ${\rm Fix}(K_X + a^* P')\geq \lceil N_\sigma (K_X + a^* P') \rceil= \lceil N_\sigma (K_X) \rceil = \lceil G \rceil$, and hence $h^0(K_X + a^* P' - \lceil G \rceil) \neq 0$. In particular, $K_X + a^* P' - \lceil G \rceil  \sim E$ for some effective divisor $E$. Therefore, using that $G \equiv K_X$, for the effective $\mathbb R$-divisors $E$ and $\{ -G\}$ we have $E + \{ - G\} \equiv 0$. This implies that both $E = 0$ and $\{ - G\}=0$.  In particular, $G$ is an integral divisor. Hence, we may assume  that $K_X-G\sim Q + a^*P$ for some torsion divisor $Q$ and $P \in \hat A$. We have that $K_X + a^* P' -  G   \sim E$ as earlier, where $E=0$ by the same argument as before. Hence $0 \sim Q + a^* (P + P')$. This implies that $Q  \in \Pic^0(X)$. So, we may assume that $Q=0$ and then $P = -P'$.

We obtained that  $V^0(\omega _X) =\{-y \}$ where $P=P_y$.
Further, we have $V^0(S^0a_*\omega _X)= V^0(\omega _X) \subset \{-y \}$.
It follows that the support of $\mathcal H ^0(\Lambda _0)$ is contained in $y$ where $\Lambda _0=R\hat S(D_A(S^0a_*\omega _X))$. By \eqref{p-lambda},  the image of $\mathcal H^0(\Lambda _0)\to \Lambda$ is a finite union of torsion translates of subtori of $\hat A$. Note that this image is not $0$ as otherwise $\Lambda =0$ and hence $\Omega =0$ (which is impossible as $\Omega $ surjects on to $S^0a_*\omega _X$). 
Thus $y$ is a torsion point. \end{proof}             
\begin{defn}
If $X$ is smooth, projective over $k$, then the $i$-th Betti number $b_i(X)$ is defined as $\dim_{\bQ_l} H^i_{\textrm{\'et}} ( X, \bQ_l)$ for some $l \neq p$. 
\end{defn}

\begin{rem}
Note that the above definition of $b_i(X)$ is independent of the choice of $l$ by an application of \cite{Deligne}.  Furthermore, ${b_1(X)}= {2}\dim \Alb(X)$ by \cite[page 14]{Liedtke}.

\end{rem}

\begin{thm}\label{t-surj} Let $a:X\to A$ be the Albanese morphism of a smooth projective variety over $k$.
If $\kappa _S(X)=0$, then $a:X\to A$ is surjective. In particular, if $\kappa_S(X)=0$, then  $b_1(X) \leq 2 \dim X$. 
\end{thm}
\begin{proof} Let $r$ be the Calabi-Yau index of $X$. By Lemma \ref{lem:S_0_if_Kodaira_dimension_zero} there is a unique  $G \in |rK_X|$. We use the notations of \eqref{notation} with the definitions
\begin{equation*}
\Delta: =\frac {r-1}{r}G, \quad \Omega _e :=F _*^{e}S^0 a_*( \sigma(X,\Delta ) \otimes \sO_X(r K_X) ) \quad \textrm{and} \quad \Omega :=\varprojlim \Omega _e. 
\end{equation*}
 Note that by \eqref{lem:Calabi_Yau_index_divides_p_minus_1}, $r | (p^e-1)$ for every $e \geq 0$ and hence  we may take the $s$ of \eqref{lem:S_f_*_mK_X} (or of \eqref{notation}) to be $1$. Also, according to \eqref{lem:S_f_*_mK_X}, $\Omega_{e+1} \to \Omega_e$ is surjective for every integer $e \geq 0$. Further, by  \eqref{lem:S_0_if_Kodaira_dimension_zero},  $S^0 a_*( \sigma(X,\Delta ) \otimes \sO_X(r K_X) ) \neq 0$. Therefore, for every $e$, $\Omega_e \ne 0$ and hence $\Omega \neq 0$.

We claim that since $\kappa (X)=0$,  \emph{there is a neighborhood $U$  of the origin such that $$V^0( \omega _X ^r)\cap U=\{ P\in \hat A|h^0(X,\omega _X ^r \otimes a^*P)\ne 0 \} \cap U=\{0 _{\hat {A}} \}.$$ } Suppose that this is not the case and let $ T \subset \hat A$ be a positive dimensional irreducible component of $V^0( \omega _X ^r)$ through the origin. 
Let $\xi : T^{g+1} \to A$ be the natural morphism. By dimension count, every fiber of $\xi$ is positive dimensional. Further, since $0 \in T$, $0$ is in the image of $\xi$.  Consider then for $(P_1,\ldots,P_{b+1}) \in \xi^{-1}(\sO_A)$ the maps $$H^0(X,\omega _X^r\otimes a^* P_1)\otimes \ldots \otimes H^0(X,\omega _X^r\otimes a^* P_{b+1})\to H^0 \left(X,\omega _X^{(b+1)r} \right).$$ Since $\dim  H^0 \left(X,\omega _X^{(b+1)r} \right)=1$, there are only finitely many choices for the divisors in $|rK_X+ a^*P_i|$  for $P_i \in p_1(\xi^{-1}(\sO_A))$.  However, since $p_1(\xi^{-1}(\sO_A))$ is an infinite set of not isomorphic line bundles and $\Pic A \to \Pic X$ is injective, this is  a contradiction. This finishes the proof of our claim. 

Since $H^0(A,S^0a_*(\sigma (X,\Delta )\otimes \omega _X^r)\otimes P)\subset H^0(X,\omega _X^r\otimes a^*P)$, we obtain that $H^0(A,\Omega_0 \otimes P)$ is zero for every $\mathcal O _A\ne P \in U$. Further, $H^0(A, \Omega_0) \neq 0$ by \eqref{lem:S_0_if_Kodaira_dimension_zero}. Hence, by \eqref{c-1}, $\sH^0(\Lambda_0)$ is Artinian in a neighborhood of the origin and $\sH^0(\Lambda_0) \otimes k(0) \neq 0$. Therefore, $\sH^0(\Lambda_0) \cong \sC_0 \oplus \sB_0$, where $\sB_0$ is Artinian and supported at $0$ and $0 \not\in \Supp \sC_0$. This induces a similar decomposition $\sC_e \oplus \sB_e$ on $\sH^0(\Lambda_e) \cong V^* \sH^0(\Lambda_0)$. Furthermore, the map $\sH^0(\Lambda_e) \to \sH^0(\Lambda_{e+1})$ is the direct sum of morphisms $\sC_e \to \sC_{e+1}$ and $\sB_e \to \sB_{e+1}$. Hence \begin{equation*}
\Lambda = \varinjlim \Lambda_e = \left(\varinjlim \sC_e \right) \oplus \left( \varinjlim \sB_e \right),                                                                                                                                                                                                                                                                                                                                                                                                                                                                                                                                                                                                                                                                                                                                                                                                                                                                                                                                                  
\end{equation*}
where $\sB_e$ is Artinian for every $e \geq 0$. 

\emph{We claim that $\varinjlim \sB_e \neq 0$.} We argue this by showing that $(\varinjlim \sB_e )\otimes k(0) = \varinjlim (\sB_e \otimes k(0))$ is not zero. To this end note that $\sB_e \otimes k(0) \cong H^0(A, F^e_* \Omega_0)^\vee$, and the homomorphism $\sB_e \otimes k(0) \to \sB_{e+1} \otimes k(0)$ can be identified via these isomorphisms with $H^0(A, F^e_* \Omega_0)^\vee \to H^0(A,F^{e+1}_* \Omega_0)^\vee$.
However, the latter is an isomorphism by  \eqref{lem:S_0_if_Kodaira_dimension_zero}. Hence $\varinjlim (\sB_e \otimes k(0)) \neq 0$ and consequently $\varinjlim \sB_e \neq 0$, which concludes the proof of our claim.

Finally, \eqref{c-2}, concludes our proof.

\end{proof}

\subsection{Proof of point $(2)$ of Theorem \ref{t-one}} \label{ss-thm_2}

\begin{thm}\label{t-bir} Let $X$ be a smooth projective variety over $k$ and $a : X \to A$ the Albanese morphism.   Then the following are equivalent:
\begin{enumerate}
\item $p \nmid \deg a$, $\kappa _S(X)=0$ and $b_1(X)=2\dim X$, and
\item $X$ is birational to an ordinary abelian variety.\end{enumerate}
\end{thm}
\begin{proof} It suffices to show that (1) implies (2). Assume that $\kappa _S(X)=0$ and $\dim X=\dim A$. Then $a$ is surjective by \eqref{t-surj} and hence it is  generically finite. We must show that the generic degree of $a$ is $1$ and $A$ is ordinary. We set $\Omega_0 = S^0 a_* \omega_X$, $\Omega_e= F^e_* \Omega_0$, $\Omega= \varprojlim \Omega_e$, $\Lambda_e = R \hat S( D_A(\Omega_e))$ and $\Lambda = \hocolim \Lambda_e$ as in \eqref{notation}.

\noindent {\bf Step 1. } \emph{$S^0( \omega _X)\ne 0$, $A$ is ordinary and the inclusion $\omega_A \to a_* \omega_X$ factors through the embedding $S^0 a_* \omega_X \to a_* \omega_X$}.

In fact, in this step, we do not need to assume that $p \nmid \deg a$, only that $a$ is separable, hence the ordinarity of $A$ holds in this more general context. Since $a$ is separable, there is a natural inclusion $\omega _A \to a_*\omega _X$ (the dual of the trace map $a_* \sO_X \to \sO_A$) inducing an inclusion $H^0(\omega _A)\subset H^0(\omega _X)$. 
Since $\kappa (X)=\kappa _S(X)=0$, this inclusion is an equality, and by \eqref{lem:S_0_if_Kodaira_dimension_zero} $S^0( \omega _X)\ne 0$.  We claim that \emph{the inclusion $\omega _A\to a_*\omega _X$ is compatible with Frobenius in the sense that the following diagram commutes}.
\begin{equation*}
\xymatrix{
F_* \omega_A \ar[d] \ar[r] & \omega_A \ar[d] \\
F_* a_* \omega_X \ar[r] & a_* \omega_X
}
\end{equation*}
Dualize the above diagram, that is apply $\mathcal{H}om_{\sO_A}(\_, \omega_A)$ to it (and Grothendieck duality at multiple places). Since dualization applied twice is the identity, it is enough to show that the dualized diagram commutes:
\begin{equation*}
\xymatrix{
F_* \sO_A  & \ar[l]  \sO_A \\
F_* a_* \sO_X \ar[u] & \ar[l] \ar[u] a_* \sO_X
} .
\end{equation*}
Because $F_* \sO_A$ is reflexive, it is further enough to show that the above diagram commutes in codimension one. That is, we may assume that $a$ is a finite map of normal varieties. Let $\phi_i$ be the embeddings of the function field of $X$ into its algebraic closure over the function field of $A$. Then we have to verify that $ \left( \sum \phi_i(f) \right)^p =  \sum \phi_i(f^p)$ holds for every local section $f$ of $\sO_X$. However, this follows since $\phi_i$ are ring homomorphisms and hence $\phi_i(f^p) = \phi_i(f)^p$. This concludes our claim.


Since the inclusion $\omega_A \to a_* \omega_X$ is compatible with Frobenius, it follows that $S^0(\omega _A)\cong S^0(\omega _X)\ne 0$.
By \eqref{p-ordinary}, this is equivalent to saying that $A$ is ordinary. Further, since $$S^0(X, \omega_X) \subseteq H^0( A, S^0 a_* \omega_X ) \subseteq H^0(A, a_*\omega_X) \cong H^0(X, \omega_X),$$ by \eqref{lem:obvious_inclusion} and since $\omega_A \cong \sO_A$, we also see that the natural morphism $\omega_A \to a_* \omega_X$ factors through $\Omega_0 = S^0 a_* \omega_X$.

\noindent {\bf Step 2. } \emph{$V^0(\omega_X)$ contains no torsion points except $\sO_A$.}

Assume the contrary, i.e., let $Q \not\cong \sO_A$ be a torsion point of $V^0(\omega_X)$ and let $D$ be the unique element of $|K_X|$.  Then considering the multiplication map (where $o(Q)$ is the order of $Q$),
\begin{equation*}
|K_X + Q| \times \dots \times |K_X + Q| \to |o(Q) K_X|
\end{equation*}
we see that the only element in $|K_X + Q|$ can be $\frac{1}{o(Q)}o(Q) D = D$. Hence $K_X + Q \sim D \sim K_X$, which is a contradiction.

\noindent {\bf Step 3.} \emph{The image of $\Lambda _0\to \Lambda$ is supported on $0_{\hat A}$.} 

If this were not the case, then by \eqref {p-lambda}, 
 there is a torsion translate $0_{\hat A}\ne T$ of an abelian subvariety such that for each $Q \in T$, the map $\Lambda_0 \otimes \sO_{\hat A, Q} \to \Lambda \otimes \sO_{\hat A, Q}$ is non-zero. Then by \eqref{c-1} also  $T \subseteq  V^0(\Omega_0) \subseteq V^0(\omega_X)$, and by Step 2 one obtains a contradiction.

\noindent {\bf Step 4.} \emph{$\sH^0(\Lambda_0) \cong k(0)$.}

Since we assume that $p \nmid \deg a$, the embedding $\omega_A \hookrightarrow a_* \omega_X$ is in fact a splitting. Further since this map factors through $\omega_A \hookrightarrow \Omega_0$, the latter also splits. In particular, $\sH^0(\Lambda_0)$ has a direct summand isomorphic to $\sH^0(R \hat S ( D_A( \omega_A))) \cong k(0)$. However, by \eqref{c-1} $\dim_k \sH^0(\Lambda_0) \otimes k(0) =1$. So, any direct complement $\sF$ of $k(0)$ in $\sH^0(\Lambda_0)$ is a coherent sheaf  supported at $0$, such that $\sF \otimes k (0)=0$. Therefore, $\sF=0$.

\noindent {\bf Step 5. } \emph{$\deg a=1$.}

By \eqref{GVT}, we know that there is a factorization 
\begin{equation*}
\xymatrix{
\Lambda_0 \ar[r] \ar@/^1pc/[rr] & \sH^0(\Lambda_0) \ar[r] & \Lambda
} .
\end{equation*}
By Steps 4, $\sH^0(\Lambda_0) \cong k(0)$. Thus 
we have a commutative diagram
\begin{equation*}
\xymatrix{
\Lambda_0 \ar[r] \ar@/^1pc/[rr] & k(0) \ar[r] & \Lambda
} .
\end{equation*}
Applying $(-1_A)^* D_{A} (RS (   \_ ))[-g]$ to the above diagram we obtain
\begin{equation*}
\xymatrix{
\Omega_0   & \sO_{A} \ar[l] & \Omega \ar@/_1pc/[ll] \ar[l]
} .
\end{equation*}
Since  the long arrow is surjective we obtain that $\rk \Omega_0 = 1$. This concludes the proof.

\end{proof}

 \end{document}